\documentclass[a4paper, 11pt]{article}
\usepackage[latin1]{inputenc}    
\usepackage[T1]{fontenc}
\usepackage{lmodern}
\usepackage{graphicx}
\usepackage{subfigure}
\usepackage{amsmath, amsfonts, amssymb, amsthm} 
\usepackage{bbold}
\usepackage{framed}
\usepackage{color}
\usepackage{tikz}
\usepackage[hidelinks]{hyperref} 
\usepackage{xcolor}
\usepackage{enumerate} 

\hypersetup{
    colorlinks,
    linkcolor={red!80!black},
    citecolor={blue!80!black},
    urlcolor={blue!80!black}
}

\usepackage{ulem}
\numberwithin{equation}{section} 

\newtheorem{define}{Definition}[section]
\newtheorem{proposition}{Proposition}[section]
\newtheorem{theorem}{Theorem}

\newtheorem{remark}{Remark}[section]
\newtheorem{lemma}{Lemma}[section]

\newcommand{\E}{\mathbb{E}}
\newcommand{\Proba}{\mathbb{P}}

\newcommand\tld[1]{\widetilde{#1}}
\newcommand\myhat[1]{\overline{#1}}

\newcommand*\myref[1]{{\normalfont (\ref{#1})}}
\newcommand*\mycite[1]{{\normalfont \cite{#1}}}

\title{Existence of martingale solutions for stochastic flocking models with local alignment}
\date{}
\author{Arnaud Debussche, Angelo Rosello {\vspace{2mm}} \\
\textit{
Univ Rennes, CNRS, IRMAR - UMR 6625, F- 35000 Rennes, France }
}

\begin{document}

\maketitle

\begin{abstract}
We establish the existence of martingale solutions to a class of stochastic conservation equations. The underlying models correspond to random perturbations of kinetic models for collective motion such as the Cucker-Smale \cite{cs1,cs2} and Motsch-Tadmor \cite{MT} models. By regularizing the coefficients, we first construct approximate solutions obtained as the mean-field limit of the corresponding particle systems. We then establish the compactness in law of this family of solutions by relying on a stochastic averaging lemma. This extends the results obtained in \cite{KMT1, KMT2} in the deterministic case.
~~
\end{abstract}

\vspace{10mm}

\tableofcontents


\clearpage

\begin{section}{Introduction, main results}

\begin{subsection}{Collective motion with local alignment}

The emergence of a consensus or ordered motion amongst a population of interacting agents has been drawing a fair amount of attention among the scientific community in recent years. This phenomenon, consistently observed in nature, from schooling fish to swarming bacteria, is usually referred to as flocking. One of the earliest and most commonly studied mathematical models describing this kind of behavior is the celebrated Cucker-Smale model, introduced in \cite{cs1, cs2}.
In this model, agents interact in a mean-field manner: for $1 \le i \le N$, denoting by $X^{i,N}, V^{i,N} \in \mathbb{R}^d$ the position and velocity of the $i$-th individual, the evolution of the system is given by
\begin{equation*}
\left\{
\begin{array}{l}
\displaystyle{ \frac{d}{dt} } X^{i,N}_t = V^{i,N}_t,
\\
\displaystyle{ \frac{d}{dt} } V^{i,N}_t = \frac{1}{N} \sum_{j=1}^N \psi(X^{i,N}_t - X^{j,N}_t) (V_t^{j,N} - V^{i,N}_t),
\end{array}
\right.
\end{equation*}
where the weight function $\psi : \mathbb{R}^d \to \mathbb{R}^+$ is even, typically of the form
\begin{align*}
\psi(x-y) = \frac{\lambda}{(1 + |x-y|^2)^\gamma}, \hspace{5mm} \lambda, \gamma > 0.
\end{align*}
Equivalently, one may consider the conservation equation
\begin{align}
\partial_t f + v \cdot \nabla_x f + \nabla_v \cdot ( L^{CS}[f] f ) = 0
\label{CS_PDE}
\end{align}
where the Cucker-Smale alignment term $L^{CS}[f]$ is given by the convolution
\begin{align}
L^{CS}[f](x,v) = \int_{\mathbb{R}^d \times \mathbb{R}^d} \psi(x-y)(w-v) f(y,w) dy dw.
\label{chap3-cucker_smale}
\end{align}
Equation \myref{CS_PDE} is naturally associated to the particle system, since it is satisfied by the empirical measure 
$$\mu^N_t = \frac{1}{N} \sum_{i=1}^N \delta_{(X_t^{i,N}, V_t^{i,N})}$$ 
in the sense of distributions. In \cite{MT}, Motsch and Tadmor brought to light several drawbacks regarding the physical relevance of the Cucker-Smale model when confronted to strongly non-homegeneous distributions of agents, due to the normalizing constant $\frac{1}{N}$ in the alignment force, which involves the whole group of individuals. To remedy these issues, they proposed a new model where the influence between two agents is normalized by the total influence:
\begin{equation*}
\left\{
\begin{array}{l}
\displaystyle{ \frac{d}{dt} } X^{i,N}_t = V^{i,N}_t,
\\
\displaystyle{ \frac{d}{dt} } V^{i,N}_t = \frac{1}{\sum_{j=1}^N \phi(X^{i,N}_t - X^{j,N}_t)} \sum_{j=1}^N \phi(X^{i,N}_t - X^{j,N}_t) (V_t^{j,N} - V^{i,N}_t).
\end{array}
\right.
\end{equation*}
Considering some weight function $\phi : \mathbb{R}^d \to \mathbb{R}^+$ with compact support, we may naturally consider a hybrid model, letting the Cucker-Smale forcing dictate the long-range interaction and the Motsch-Tadmor term dictate the short-range interaction:
\begin{align}
\partial_t f + v \cdot \nabla_x f + \nabla_v \cdot ( (L^{CS}[f] + L^{MT}[f] ) f ) = 0
\label{MT_PDE}
\end{align}
where the $L^{MT}[f]$ is given by
\begin{align}
L^{MT}[f](x,v) =  \frac{ 
\displaystyle{
\int_{\mathbb{R}^d \times \mathbb{R}^d} \phi(x-y) (w-v) \, f(y,w) dy dw } }
{ \displaystyle{ \int_{\mathbb{R}^d \times \mathbb{R}^d} \phi(x-y) f(y,w) dy dw }  } 
.
\label{motsch_tadmor}
\end{align}
Note that the Motsch-Tadmor forcing term can also be written 
$$L^{MT}[f](x,v) = u[f](x) - v,$$
expressing the alignment of the speed with the local average velocity $u[f]$, defined as
\begin{align*}
& u[f](x) = \frac{ 
\displaystyle{
\int_{\mathbb{R}^d \times \mathbb{R}^d} \phi(x-y) w \, f(y,w) dydw } }
{ \displaystyle{ \int_{\mathbb{R}^d \times \mathbb{R}^d} \phi(x-y) f(y,w) dy dw }  }.
\end{align*}
As suggested in \cite{KMT1}, we may also consider \myref{MT_PDE} in the singular limit where the weight function $\phi$ governing the short-range interaction converges to the Dirac function $\delta_0$, leading to
\begin{align}
\partial_t f + v \cdot \nabla_x f + \nabla_v \cdot ( (L^{CS}[f] + L^{SLA}[f] ) f ) = 0.
\label{SLA_PDE}
\end{align}
In \myref{SLA_PDE}, the Strong Local Alignment term $L^{SLA}[f]$ is given by
\begin{align}
L^{SLA}[f](x,v) = u_0[f](x) - v,
\label{sla}
\end{align}
where the local velocity $u_0[f]$ is given by
\begin{align*}
u_0[f](x) = \displaystyle{ \frac{ \int_{\mathbb{R}^d} w f(x,w) dw}{\int_{\mathbb{R}^d} f(x,w) dw   }}.
\end{align*}
The existence of solutions to the kinetic equations \myref{MT_PDE} and \myref{SLA_PDE} has been established in \cite{KMT1}. Moreover, in \cite{KMT2}, the authors rigorously explore the limit $\phi \to \delta_0$: considering some $\phi_1 \in C_c(\mathbb{R}^d)$ and weight functions of the form $\phi^r(x) = r^{-d} \phi_1(x/r)$, solutions $(f^r)_{r \ge 0}$ of \myref{MT_PDE} converge (up to some subsequence) to a solution $f$ of \myref{SLA_PDE}.

In order to take into account random phenomena emerging from the environment, or unpredictable interactions between the agents, it is rather natural to perturb the deterministic equations \myref{MT_PDE} and \myref{SLA_PDE} with some noise, driven by a Wiener process $dW(z) = \sum_k K_k[f](z) d\beta^k_t$, leading to the stochastic conservation equation
$$
d f_t + \Big[ v \cdot \nabla_x f_t + \nabla_v \cdot ( (L[f_t] f_t ) \Big] dt + 
\sum_k \nabla_v \cdot ( K_k[f_t] f_t) \circ d\beta_t^k
= 0
$$
with $L[f] = L^{CS}[f] + L^{MT}[f]$ or $L[f] = L^{CS}[f] + L^{SLA}[f]$.
Here, for simplicity purposes, we choose to only consider a "one-dimensional" noise driven by a real valued Brownian motion $\beta = (\beta_t)_{t \ge 0}$, leading to equations
\begin{align}
d f_t + \Big[ v \cdot \nabla_x f_t + \nabla_v \cdot ( (L^{CS}[f_t] + L^{MT}[f_t] ) f_t ) \Big] dt + 
\nabla_v \cdot ( K[f_t] f_t) \circ d\beta_t
= 0
\label{chap3-spde}
\end{align}
and
\begin{align}
d f_t + \Big[ v \cdot \nabla_x f_t + \nabla_v \cdot ( (L^{CS}[f_t] + L^{SLA}[f_t] ) f_t ) \Big] dt + 
\nabla_v \cdot ( K[f_t] f_t) \circ d\beta_t
= 0.
\label{spde2}
\end{align}
The methods developed in the present paper shall not rely on this particular form, so that the results may be easily generalized to SPDEs with multiple Brownian motions.
Note that these stochastic conservation equations are written in Stratonovitch form, since it is the most physically relevant form. As in the deterministic case, equation  \myref{chap3-spde} is naturally associated to the stochastic particle system
\begin{equation}
\left\{
\begin{array}{l}
dX^{i,N}_t = V^{i,N}_t dt,
\\
dV^{i,N}_t = L[\mu^N_t] (X^{i,N}_t, V^{i,N}_t) dt + K[\mu^N_t](X^{i,N}_t, V^{i,N}_t) \circ d\beta_t
\end{array}
\right.
\label{sto_particle}
\end{equation}
where $\mu^N_t = \frac{1}{N} \sum_{i=1}^N \delta_{(X^{i,N}_t , V^{i,N}_t)}$ and $L[\mu] = L^{CS}[\mu] + L^{MT}[\mu]$.
\vspace{3mm}

The mean-field convergence of \myref{sto_particle} to the corresponding limiting SPDE has been studied in the litterature in the case of the Cucker-Smale interaction, that is with $L[\mu] = L^{CS}[\mu]$. The diffusion coefficient
 $K[\mu](x,v)~=~D(v-v_c)$ for some constant $v_c \in \mathbb{R}^d$ is considered in \cite{ahn_mi_ha} ; the coefficient $K[\mu](x,v) = \sqrt{2\sigma}(v - \bar v)$, where $\bar v = \int w d\mu(y,w)$, is looked upon in \cite{choi-salem} and \cite{jung-ha} ; some more general (non linear) diffusion coefficients are considered in \cite{rosello}.
 
As for the Motsch-Tadmor model, that is for $L[\mu] = L^{MT}[\mu]$, the flocking phenomenon for the particle system \myref{sto_particle} (alignment of speeds, distance between the individuals bounded over time) is studied in \cite{liu} in the case of a multiplicative noise $K[\mu](x,v) = Dv$.
However, even in the deterministic case, due to the singular ratio involved in the non-linear term $L^{MT}[\mu]$, the mean-field limit of the Motsch-Tadmor particle system is a very delicate question (as suggested in \mycite{MT} and \mycite{consensus}) to which the authors could not find a proper answer in the litterature. It should also be noted that the strong local alignment term $L^{SLA}[f]$ given by \myref{sla} is ill-defined when $f$ is a general measure, so that the particle system associated with equation \myref{spde2} cannot in fact be written.
\vspace{3mm}

In the present work, we shall consider a diffusion coefficient of the form
\begin{align}
K[f](x,v) = F(x) + \int_{\mathbb{R}^d \times \mathbb{R}^d} \tld{\psi}(x-y) (w-v) f(y,w) dy dw
\label{K}
\end{align}
which corresponds to some random environmental forcing $F(x) \circ d \beta_t$, as well as a random perturbation $\psi \leftarrow \psi + \tld \psi \circ d\beta_t$ of the weight function involved in the Cucker-Smale alignment term \myref{chap3-cucker_smale} (as considered in \mycite{choi-salem} and \mycite{rosello} in the case of the Cucker-Smale model only). Note that the choice $F=0$ and $\tld \psi = \sqrt{2\sigma}$ leads in particular to the simpler coefficient $K[\mu](x,v) = \sqrt{2\sigma}(v-\bar v)$. The arguments developed in the present work also easily apply to the case $K[\mu](x,v) = Dv$ looked upon in \mycite{liu}.

In this paper, we extend the work developed in \cite{KMT1} and \cite{KMT2} in the deterministic case to establish the existence of martingale solutions (see Definition \ref{martingale_def} below) for the stochastic conservation equations \myref{chap3-spde} and \myref{spde2}. To this intent, we start by regularizing the coefficients: in section \ref{sec2}, we prove the existence of a unique solution of equation \myref{chap3-spde} with regularized coefficients, which is naturally constructed as the mean-field limit of the corresponding stochastic particle system. Then, in section \ref{sec3}, we prove the tightness of these approximate solutions with respect to the regularizing parameter, and rigorously pass to the limit in the martingale problem associated with \myref{chap3-spde}.

\end{subsection}

\begin{subsection}{Assumptions and main results}
The weight functions $\psi, \tld \psi : \mathbb{R}^d \to \mathbb{R}^+$ involved in \myref{chap3-cucker_smale} and \myref{K} are assumed to satisfy
\begin{align}
& |\psi(x)| + |\tld \psi(x)| \lesssim 1, \hspace{5mm} 
|\partial^\alpha_x \psi(x)| + |\partial^\alpha_x \tld \psi(x)| \lesssim 1 \text{ for $1 \le |\alpha| \le 4$}.
\label{psi}
\end{align}
The weight function $\phi : \mathbb{R}^d \to \mathbb{R}^+$ involved in \myref{motsch_tadmor} is assumed to be smooth and compactly supported around zero: $\phi \in C^\infty(\mathbb{R}^d)$ and for some $0 < r_1 < r_2 < \infty$,
\begin{align}
\inf_{x \in B(0,r_1)} \phi(x) > 0, \hspace{10mm} Supp(\phi) \subset B(0,r_2) .
\label{phi}
\end{align}
The forcing $F$ involved in \myref{K} is assumed to be smooth and sublinear:
\begin{align*}
|F(x)| \lesssim 1 + |x|.
\end{align*}
Simple calculations show that the proper Itô form associated with SPDE \myref{chap3-spde} is
\begin{align*}
df_t + \nabla_v \cdot \Big( \Big( L^{CS}[f_t] + L^{MT}[f_t] + S[f_t] \Big) f_t \Big) dt 
+ \nabla_v \cdot \left( K[f_t] f_t \right)  d\beta_t = \frac{1}{2} \nabla_v^2 (K[f_t]K[f_t]^T f_t) dt 
\end{align*}
where we have used the notation
$$
\nabla_v^2 (K[f] K[f]^T f) = \sum_{1 \le i, j \le d} \partial^2_{v_i v_j} \Big[ K[f]^i K[f]^j f \Big]
$$
and the additional drift forcing term $S[f]$ is given by
\begin{align*}
& S[f](z) = \frac{1}{2} \int_{\mathbb{R}^{2d}}
\tld \psi(x-y) \Big( K[f](y,w) - K[f](x,v) \Big) f(y,w) dydw.
\end{align*}
This motivates the following definition.

\begin{define} \label{martingale_def}
Let $T > 0$ and  $(\Omega, {\cal F}, ({\cal F}_t)_{0 \le t \le T}, \Proba, \beta)$ be a filtered probability space equipped with an $({\cal F}_t)$-brownian motion $\beta$. Let $f_0 : \mathbb{R}^{2d} \to \mathbb{R}_+$ with $\int_{\mathbb{R}^{2d}} f_0(z) dz = 1$. \\
A process $f \in L^\infty \Big( \Omega ; L^\infty([0,T] ; L^1(\mathbb{R}^{2d}))\Big)$ with
\begin{align}
\Big[\int_{\mathbb{R}^{2d}} | f(\omega)(t,z) | dz \le 1, \; \; dt\text{-a.e} \Big], \; \; \Proba\text{-a.s} 
\label{mart_est1}
\end{align}
satisfying the estimate
\begin{align}
\E \Big[ \int_0^T \int_{\mathbb{R}^{2d}} (1 + |v|^2) | f(t,z) |  dz dt \Big] < \infty
\label{mart_est2}
\end{align}
is said to be a solution of \myref{chap3-spde} on $(\Omega, \beta)$ (with initial value $f_0$) when, for any test function $\Psi \in C^\infty_c(\mathbb{R}^{2d})$, the process $(\langle f(t), \Psi \rangle)_{0 \le t \le T}$ is adapted with a continuous version and satisfies
$$
\langle f(t) , \Psi \rangle = \langle f_0, \Psi \rangle + \int_0^t \bigl< {\cal L}[f(s)] \Psi, f(s) \bigr> ds
+ \int_0^t \bigl< K[f(s)] \cdot \nabla_v \Psi, f(s) \bigr> d\beta_s, \; \; \; t \in [0,T]
$$
where ${\cal L}[f]$ denotes the second order operator
\begin{align}
{\cal L}[f] \Psi = v \cdot \nabla_x \Psi +  \Big( L^{CS}[f] + L^{MT}[f] + S[f] \Big) \cdot \nabla_v \Psi 
+ \frac{1}{2} \sum_{1 \le i, j \le d} K[f]^i K[f]^j \partial^2_{v_i v_j} \Psi.
\label{limit_operator}
\end{align}
If there exists some probability tuple $(\myhat \Omega, \myhat {\cal F}, (\myhat {\cal F}_t)_{0 \le t \le T}, \myhat \Proba, \myhat \beta)$ and a solution $\myhat f$ of \myref{chap3-spde} on $(\myhat \Omega, \myhat \beta)$, we say that equation $\myref{chap3-spde}$ has a martingale solution (in the sense of {\normalfont \cite{da_prato}},  Chapter 8).
\end{define}
Estimates \myref{mart_est1} and \myref{mart_est2} are quite natural since we expect solutions of \myref{chap3-spde} to be densities. On can easily check that these estimates guarantee that the process $(\langle f(t), \Psi \rangle)_{t \ge 0}$ is well defined, the stochastic integral being a square integrable martingale.
Solutions of equation \myref{spde2} are defined similarly. We now state our main results.

\begin{theorem}[Stochastic Motsch-Tadmor flocking] \label{chap3-thm1} ~~\\
Let $f_0 : \mathbb{R}^{2d} \to \mathbb{R}_+$ with $\int_{\mathbb{R}^{2d}} f_0(z) dz = 1$
such that, for some $\delta > 1$ and $\theta \in (0,1)$,
\begin{align}
\int_{\mathbb{R}^{2d}} |f_0(z)|^p dz + \int_{\mathbb{R}^{2d}} ( |x|^\delta + |v|^k ) f_0(z) dz < \infty,
\hspace{5mm} p = 1 + \frac{1}{\theta}, \; \; k > \frac{ \max(d+2,4) }{1- \theta}.
\label{hyp}
\end{align}
Then there exists a martingale solution $f$ of equation \myref{chap3-spde} with initial data $f_0$.
\end{theorem}

We also prove the existence of a martingale solution of \myref{spde2}, which can be constructed as a weak limit of solutions of \myref{chap3-spde} as the function $\phi$ involved in the Motsch-Tadmor alignment term \myref{motsch_tadmor} properly approaches the Dirac function $\delta_0$, similarly to the result established in \cite{KMT2}.

\begin{theorem}[Stochastic flocking with strong local alignment] \label{chap3-thm2} ~~\\
Let $\phi_1 \in C_c^\infty(\mathbb{R}^d)$ satisfy \myref{phi}. Let us consider the sequence of functions $(\phi_r)_{r > 0}$ given by
$$
\forall x \in \mathbb{R}^d, \; \; \; \phi_r(x) = r^d \phi_1(x/r).
$$
Assuming \myref{hyp}, the martingale solution $f^r$ of equation \myref{chap3-spde} with $\phi = \phi_r$ constructed in Theorem \ref{chap3-thm1} satisfies, along some subsequence $r_n \to 0$,
$$
f^{r_n} \to f \text{ in law}, \text{ in } C([0,T] ; H^{-\sigma}_{W^{-1}}(\mathbb{R}^{2d}))
\text{ for all $\sigma > 0$},
$$
where $f$ defines a martingale solution of \myref{spde2}. The weighted Sobolev space $H^{-\sigma}_{W^{-1}}(\mathbb{R}^{2d})$, with $W^{-1}(z) = (1 + |z|)^{-1}$ is introduced in \myref{weight2} below.
\end{theorem}

Note that, although $f$ only converges weakly in space, the use of a stochastic averaging lemma (developed in Proposition \ref{averaging_lemma} below) will guarantee that integrated quantities of the form
$$
\rho_\varphi = \int_{\mathbb{R}^d} \varphi(v) f dv
$$
converge in the strong space $L^2([0,T] ; L^2(\mathbb{R}^{d}))$. This allows to properly pass to the limit in the non-linear equation \myref{chap3-spde}.

\end{subsection}

\begin{subsection}*{Acknowledgment}

A. Debussche and A. Rosello are partially supported by the French government thanks to the "Investissements d'Avenir"
program ANR-11-LABX-0020-01.

\end{subsection}

\end{section}

\begin{section}{Regularized equation} \label{sec2}

In this section, we prove  existence and uniqueness of a solution of equation \myref{chap3-spde} with regularized coefficients. This solution will be naturally obtained as the mean-field limit of the corresponding particle system.
\vspace{3mm}

For $R > 0$ let us introduce smooth, compactly supported truncation functions 
$$
\chi_R \in C^\infty_c(\mathbb{R}^d ; \mathbb{R}), \hspace{5mm} \theta_R \in C^\infty_c(\mathbb{R}^d ; \mathbb{R}^d)
$$
satisfying
\begin{equation}
\left\{
\begin{array}{l}
 \chi_R(x) = 1 \; \text{ if } \; |x| \le R, \hspace{10mm} |\chi_R(x)| \le 1,
\vspace{1mm} \\
 \theta_R(v) = v  \; \text{ if }  \; |v| \le R, \hspace{10mm} |\theta_R(v)| \le |v|,
\vspace{1mm} \\
  | \nabla_v \cdot \theta_R(v) | \lesssim 1 \; \; \text{ uniformly in } R > 0. 
\end{array}
\right.
\label{theta}
\end{equation}
We may then introduce the following regularized coefficients:
\begin{equation}
\left\{
\begin{array}{l}
 L^{CS}_R[\mu](z) = \displaystyle{ \int  \chi_R(x-y) \psi(x-y) \theta_R(w-v) d\mu(y,w),  }
\vspace{2mm} \\
u_{R} [\mu](x) = \frac{ 
\displaystyle{
\int_{\mathbb{R}^d \times \mathbb{R}^d} \phi(x-y)  \theta_R(w) \, d\mu(y,w) } }
{  \displaystyle{ R^{-1} + \int_{\mathbb{R}^d \times \mathbb{R}^d} \phi(x-y) d\mu(y,w) }  }, 
\vspace{2mm} \\
 L^{MT}_R[\mu](z) = u_{R}[\mu](x) - v,
\vspace{2mm} \\
 K_R[\mu](z) = \displaystyle{ \chi_R(x) F(x) + 
\int_{\mathbb{R}^d \times \mathbb{R}^d} \chi_R(x-y) \tld \psi(x-y) \theta_R(w-v) d\mu(y,w),  }
\vspace{2mm} \\
 S_R[\mu](z) = \displaystyle{ \frac{1}{2} \int_{\mathbb{R}^{2d}} 
\tld \psi(x-y) \Big( K_R[\mu](y,w) - K_R[\mu](x,v) \Big) d\mu(y,w). }
\end{array}
\right.
\label{reg}
\end{equation}
Simple calculations show that, for fixed $R$, these regularized coefficients are globally Lipschitz continuous in the following sense: for all $z, z' \in \mathbb{R}^{2d}$, $\mu, \nu \in {\cal P}(\mathbb{R}^{2d})$,
\begin{equation}
\begin{array}{l}
\Big| L^{CS}_R[\mu](z) - L^{CS}_R[\nu](z') \Big| + 
\Big| L^{MT}_R[\mu](z) - L^{MT}_R[\nu](z') \Big| 
\lesssim |z-z'| + W_1[\mu,\nu],
\\
\Big| K_R[\mu](z) - K_R[\nu](z') \Big| + \Big| S_R[\mu](z) - S_R[\nu](z') \Big|
\lesssim |z-z'| + W_1[\mu,\nu]
\end{array}
\label{glob_lip}
\end{equation}
where the constants involved in $\lesssim $ in \myref{glob_lip} depend on $R$, and $W_1$ denotes the Wasserstein distance. The assumptions made also guarantee the uniform sub-linearity of some coefficients: for all $z \in \mathbb{R}^{2d}$, $\mu \in {\cal P}(\mathbb{R}^{2d})$,
\begin{align}
\Big| L^{CS}_R[\mu](z) \Big| + \Big| K_R[\mu](z) \Big| + \Big| S_R[\mu](z) \Big| \lesssim 1 + |z| + \int |z'| d\mu(z')
\label{sublinear}
\end{align}
where the constant involved in $\lesssim$ in \myref{sublinear} does not depend on $R$.

\begin{subsection}{Mean-field limit of the associated particle system}

We may now consider the associated mean-field particle system on $Z^{i,N} = (X^{i,N}, V^{i,N}) \in \mathbb{R}^{2d}$, $i\in \{1, \ldots, N \}$:
\begin{equation}
\left\{
\begin{array}{l }
dX^{i,N}_t  =  V^{i,N}_t dt
\\
dV^{i,N}_t  =  \Big( L^{CS}_R[\mu^N_t] + L^{MT}_R[\mu^N_t] + S_R[\mu^N_t] \Big)(Z^{i,N}_t) dt  + K_R[\mu^N_t](Z^{i,N}_t) d\beta_t,
\end{array}
\right.
\label{part_reg}
\end{equation}
where $\mu_t^N$ denotes the empirical measure
\begin{align*}
\mu_t^N = \frac{1}{N} \sum_{i=1}^N \delta_{Z^{i,N}_t}.
\end{align*}
From the sub-linearity of the coefficients, we easily deduce the following result.

\begin{proposition} \label{estim1}
For any $T > 0$ and $z_0^{i,N} \in \mathbb{R}^{2d}$, $i \in \{1, \ldots, N \}$, the SDE system \myref{part_reg} with initial condition $Z_0^{i,N} = z_0^{i,N}$ has a unique global solution $(Z^{i,N}_t)_{t \in [0,T]}^{i \in \{1, \ldots, N \}}$ which satisfies the following estimates: for any $p \ge 1$,
\begin{align*}
& \mathbb{E} \Big[ \sup_{t \in [0,T]} \int_{\mathbb{R}^{2d}} |z|^p d\mu^N_t(z) \Big] \lesssim 1 + \int_{\mathbb{R}^{2d}} |z|^p d\mu_0^N(z),
\\
& \mathbb{E} \Big[ \sup_{t \in [0,T]} \Big| Z^{i,N}_t \Big|^p \Big] \lesssim 1+ |z_0^{i,N}|^p +  \int_{\mathbb{R}^{2d}} |z|^p d\mu_0^N(z).
\end{align*}
The constants involved in $\lesssim$ depend on $R, p$ and $T$ only.
\end{proposition}
\begin{proof}[Proof]
The coefficients of \myref{part_reg} being locally Lipschitz-continuous, the local existence and uniqueness of solutions is guaranteed.
The estimates of Proposition \ref{estim1} should first be established with the stopping time 
$$ \tau_M = \inf \left\{ t \ge 0, \; \; \max_{1 \le i \le N} \Big| Z^{i,N} \Big| \ge M \right\} \wedge T,$$
which should then be sent to $T$, as $M$ goes to infinity. For the sake of simplicity, we omit this stopping time in the following. Given the sub-linearity of the coefficients, Itô's formula gives
\begin{align}
\Big| Z^{i,N}_t \Big|^p = |z_0^{i,N} |^p + \int_0^t A^i_s ds + \int_0^t B^i_s d\beta_s
\label{ito1}
\end{align}
with
\begin{align*}
A^i_t \lesssim 1 + |Z^{i,N}_t|^p + \int |z|^p d\mu^N_t, \hspace{10mm}
B^i_t = p \Big|Z_t^{i,N} \Big|^{p-2} Z_t^{i,N} \cdot K_R[\mu_t^N](Z_t^{i,N}).
\end{align*}
Averaging \myref{ito1} over $i$, we are led to
\begin{align}
\int |z|^p d\mu^N_t = \int |z|^p d\mu_0^N + \int_0^t \Big( \frac{1}{N} \sum_{i=1}^N A^i_s  \Big) ds 
+ \int_0^t \Big( \frac{1}{N} \sum_{i=1}^N B^i_s \Big) d\beta_s
\label{ito2}
\end{align}
from which we easily deduce, for all $0 \le r \le T$
\begin{align*}
\E \Big[ \sup_{t \in [0,r]} \int |z|^p d\mu^N_t \Big] 
 \lesssim 1 +  \E \Big[ \int |z|^p d\mu_0^N \Big]  
& +  \int_0^r \E \Big[ \sup_{t \in [0,s]} \int |z|^p d\mu_t^N \Big]  ds 
\\
&
+ \E \Big[ \sup_{t \in [0,r]}  \int_0^t \Big( \frac{1}{N} \sum_{i=1}^N B^i_s \Big) d\beta_s \Big].
\end{align*}
Burkholder-Davis-Gundy's inequality (from \cite{bdg}) gives
\begin{align*}
\E \Big[ \sup_{t \in [0,r]}  \int_0^t \Big( \frac{1}{N} \sum_{i=1}^N B^i_s \Big) d\beta_s \Big] 
& \lesssim
\E \Big[ \Big( \int_0^r \Big| \frac{1}{N} \sum_{i=1}^N B^i_s \Big|^2 ds \Big)^{1/2} \Big]
\lesssim \E \Big[ \Big( \int_0^r \Big| 1 + \int |z|^p d\mu_s^N \Big|^2 ds \Big)^{1/2} \Big]
\\
& \lesssim 1 + \E \Big[ \Big( \sup_{t \in [0,r]} \int |z|^p d\mu_t^N  \Big)^{1/2} \Big( \int_0^r \int |z|^p d\mu_s^N ds \Big)^{1/2}  \Big]
\\
& \lesssim 1 + \E \Big[ \sup_{t \in [0,r]} \int |z|^p d\mu_t^N   \Big]^{1/2} \int_0^r \E \Big[ \sup_{t \in [0,s]} \int |z|^p d\mu^N_t \Big]^{1/2}
\end{align*}
so that we may come back to \myref{ito2} and get
\begin{align*}
\E \Big[ \sup_{t \in [0,r]} \int |z|^p d\mu^N_t \Big] 
 \lesssim 1
 +  \E \Big[ \int |z|^p d\mu_0^N \Big]   +  \int_0^r \E \Big[ \sup_{t \in [0,s]} \int |z|^p d\mu_t^N \Big]  ds.
\end{align*}
We may now apply Gronwall's lemma to derive the first estimate. Coming back to \myref{ito1}, a similar reasoning leads to the second estimate.

\end{proof}

Let us now consider the space of trajectories 
$${\cal C} = C([0,T] ; \mathbb{R}^{2d}), \hspace{5mm} \| z \|_\infty = \sup_{t \in [0,T]} |z_t|$$ 
and view the empirical measure as a (random) probability over ${\cal C}$: 
\begin{align*}
\mu^N = \frac{1}{N} \sum_{i=1}^N \delta_{Z^{i,N}} \in {\cal P}({\cal C}).
\end{align*}
More precisely, for $p \ge 1$, we introduce the Wasserstein space
$$
{\cal P}_p({\cal C}) = \left\{ \mu \in {\cal P}({\cal C}), \; \; \int_{z \in {\cal C}} \| z \|_\infty^p d\mu < \infty \right\}
$$
equipped with the usual distance
$$
W_p[\mu,\nu] = \inf_{\pi \in \Pi(\mu,\nu)} \Big( \int_{(z_1,z_2) \in {\cal C}^2} \| z_1 - z_2 \|_\infty^p d\pi(z_1,z_2) \Big)^{1/p}
$$
where 
$$\Pi(\mu,\nu) = \left\{ \pi \in {\cal P}({\cal C}^2), \; \; \int_{z_2 \in {\cal C}} \pi( dz_1, dz_2) = \mu(dz_1), \; 
\int_{z_1 \in {\cal C}} \pi( dz_1, dz_2) = \nu(dz_2)
 \right\}.$$ 
 We may now state the following mean-field limit result.
 
\begin{proposition} \label{mean_field} Let $p > 1$.
As $N \to \infty$, provided that $\mu_0^N \to \mu_0$ in ${\cal P}_p(\mathbb{R}^{2d})$, we have 
$\mu^N \to \mu \text{ in } L^p(\Omega ; {\cal P}_p({\cal C}))$
where $\mu$ solves the regularized SPDE
\begin{align}
& d\mu_t + v \cdot \nabla_x \mu_t dt +  \nabla_v \cdot \Big( \Big( L^{CS}_R[\mu_t] + L^{MT}_{R}[\mu_t] + S_R[\mu_t] \Big) \mu_t \Big) dt 
+ \nabla_v \cdot \left( K_R[\mu_t] \mu_t \right) d\beta_t 
\nonumber \\ & \hspace{60mm}
= \frac{1}{2} \nabla_v \cdot \Big( \nabla_v \cdot ( K_R[\mu_t] K_R[\mu_t]^T \mu_t ) \Big) dt
\label{spde_reg}
\end{align}
in the following sense: denoting the operator
\begin{align}
{\cal L}_R[\mu] \Psi = v \cdot \nabla_x \Psi +  \Big( L_R^{CS}[\mu] + L_R^{MT}[\mu] + S_R[\mu] \Big) \cdot \nabla_v \Psi + \frac{1}{2} \sum_{1 \le i, j \le d} K_R[\mu]^i K_R[\mu]^j \partial^2_{v_i v_j} \Psi
\label{dual_operator}
\end{align}
we have, for any test function $\Psi \in C_c^\infty(\mathbb{R}^{2d})$, 
\begin{align*}
\langle \mu_t, \Psi \rangle = \langle \mu_0, \Psi \rangle + 
\int_0^t \Bigl< {\cal L}_R[\mu_s] \Psi , \mu_s \Bigr> ds +
\int_0^t \Bigl< K_R[\mu_s] \cdot \nabla_v \Psi , \mu_s \Bigr> d\beta_s, \; \; \; t \in [0,T].
\end{align*}
More precisely, $\mu$ is the unique element of $L^2(\Omega ; {\cal P}_p({\cal C}))$  given by the push-forward measure of the initial data by the non-linear characteristics:
\begin{align}
\mu = (Z^\mu)^* \mu_0 \text{ in } {\cal P}({\cal C}), \; \; a.s \label{fix_point}
\end{align}
where $Z^\mu : z \in \mathbb{R}^{2d} \mapsto (X^\mu_t(z), V^\mu_t(z))_{t \in [0,T]} \in {\cal C}$ is the flow associated with the SDE
\begin{equation}
\left\{
\begin{array}{l}
dX^\mu_t(z) = V^\mu_t(z) dt,
\\
dV^\mu_t(z) =  \Big( L^{CS}_R[\mu_t]+ L^{MT}_R[\mu_t] + S_R[\mu_t] \Big)(Z^\mu_t(z))  dt + K_R[\mu_t](Z^{\mu}_t(z)) d\beta_t,
\\
Z^\mu_0(z) = z.
\end{array}
\right.
\label{chap3-chara}
\end{equation}

\end{proposition}

\begin{proof}
For simplicity, let us consider the case $p=2$.
We start by noticing that, for fixed $N \ge 1$, the empirical measure $\mu^N$ naturally satisfies the fixed-point-like equation \myref{fix_point}. For all $N, M \ge 1$, introducing an optimal plan  $\pi \in \Pi(\mu_0^N, \mu_0^M)$ (one may refer to \cite{villani} for details) such that
$$
W_2^2[\mu_0^N, \mu_0^M] = \int_{(z_1,z_2) \in (\mathbb{R}^{2d})^2} | z_1 - z_2 |^2 d\pi(z_1,z_2),
$$
it follows that
\begin{align}
W_2^2[\mu^N, \mu^M] \le  \int_{(z_1,z_2) \in (\mathbb{R}^{2d})^2} \| Z^{\mu^N}(z_1) - Z^{\mu^M}(z_2) \|^2_\infty d\pi(z_1,z_2) =: J^{N,M}_T
\label{JNM1}
\end{align}
and we may simply re-write
\begin{align}
J^{N,M}_T = \sum_{1 \le i \le N} \sum_{1 \le j \le M} \sup_{t \in [0,T]} \Big| Z^{i,N}_t - Z^{j,M}_t \Big|^2 \pi(\{z_0^{i,N}, z_0^{j,M} \})
\label{JNM}
\end{align}
where $Z^{i,N}$ is the solution of \myref{part_reg}. Itô's formula easily leads to
\begin{align*}
d \Big| Z^{i,N}_t - Z^{j,M}_t \Big|^2 = \Big( \zeta^1_t + \zeta^2_t + \zeta^3_t + \zeta^4_t  \Big) dt + \zeta^5_t d\beta_t
\end{align*}
where
\begin{align*}
& \zeta^1_t =  2 \Big( Z^{i,N}_t - Z^{j,M}_t \Big) \cdot \Big( L^{CS}_R[\mu_t^N](Z^{i,N}_t) - L^{CS}_R[\mu_t^M](Z^{j,M}_t) \Big),
\\
& \zeta^2_t =
2 \Big( Z^{i,N}_t - Z^{j,M}_t \Big) \cdot \Big( L^{MT}_R[\mu_t^N](Z^{i,N}_t) - L^{MT}_R[\mu_t^M](Z^{j,M}_t) \Big),
\\
& \zeta^3_t = 
2 \Big( Z^{i,N}_t - Z^{j,M}_t \Big) \cdot \Big(  S_R[\mu_t^N](Z^{i,N}_t) - S_R[\mu_t^M](Z^{j,M}_t) \Big),
\\
& \zeta^4_t =  \Big| K_R[\mu_t^N](Z^{i,N}_t) - K_R[\mu_t^M](Z^{j,M}_t) \Big|^2,
\\
& \zeta^5_t = 
2 \Big( Z^{i,N}_t - Z^{j,M}_t \Big) \cdot \Big(  K_R[\mu_t^N](Z^{i,N}_t) - K_R[\mu_t^M](Z^{j,M}_t) \Big).
\end{align*}
Using the Lipschitz estimate \myref{glob_lip}, we deduce (for fixed $R > 0$)
\begin{align}
d \Big| Z^{i,N}_t - Z^{j,M}_t \Big|^2 &  \lesssim \Big( \Big| Z^{i,N}_t - Z^{j,M}_t \Big|^2  + W_1[ \mu_t^N, \mu_t^M ]^2 \Big) dt + \zeta^5_t d\beta_t \nonumber
\\ &
\lesssim \Big( \Big| Z^{i,N}_t - Z^{j,M}_t \Big|^2  + J_t^{N,M} \Big)dt + \zeta^5_t d\beta_t.
\label{this1}
\end{align}
Taking the expectation in \myref{this1} and applying Gronwall's lemma leads to
\begin{align*}
\mathbb{E} \Big| Z^{i,N}_t - Z^{j,M}_t \Big|^2 \lesssim \Big( | z_0^{i,N} - z_0^{j,M} |^2 + \mathbb{E} [ J_t^{N,M} ]  \Big).
\end{align*}
Coming back to \myref{this1}, we may write
\begin{align*}
\sup_{\sigma \in [0,t]} \Big| Z^{i,N}_\sigma - Z^{j,M}_\sigma \Big|^2 \lesssim 
| z_0^{i,N} - z_0^{j,M} |^2 + \int_0^t \Big( \Big| Z^{i,N}_s - Z^{j,M}_s \Big|^2 + J_s^{N,M}  \Big) ds +
\sup_{\sigma \in [0,t]} \int_0^\sigma \zeta^5_s d\beta_s 
\end{align*}
and therefore
\begin{align}
\E \Big[ \sup_{\sigma \in [0,t]} \Big| Z^{i,N}_\sigma - Z^{j,M}_\sigma \Big|^2 \Big] \lesssim 
| z_0^{i,N} - z_0^{j,M} |^2 + \int_0^t \E [ J_s^{N,M} ]  ds +
\E \Big[ \sup_{\sigma \in [0,t]} \int_0^\sigma \zeta^5_s d\beta_s  \Big].
\label{this2}
\end{align}
Burkholder-Davis-Gundy's inequality gives
\begin{align*}
& \E \Big[ \sup_{\sigma \in [0,t]} \int_0^\sigma \zeta^5_s d\beta_s  \Big]  \lesssim \E \Big[ \Big( \int_0^t | \zeta^5_s|^2 ds \Big)^{1/2} \Big]
\\
& \hspace{15mm} \lesssim \E \Big[ 
\sup_{s \in [0,t]} \Big| Z^{i,N}_s - Z^{j,M}_s \Big| \Big] \Big(
 \int_0^t \Big|  K_R[\mu_s^N](Z^{i,N}_s) - K_R[\mu_s^M](Z^{j,M}_s)  \Big|^2 ds \Big)^{1/2}
   \Big]
\\
& \hspace{15mm} \lesssim \E \Big[ \sup_{s \in [0,t]} \Big| Z^{i,N}_s - Z^{j,M}_s \Big|^2 \Big]^{1/2} 
\Big(   \int_0^t \E \Big|  K_R[\mu_s^N](Z^{i,N}_s) - K_R[\mu_s^M](Z^{j,M}_s)  \Big|^2 ds
\Big)^{1/2}.
\end{align*}
Making use of \myref{glob_lip} again, we get
\begin{align*}
\E \Big[ \sup_{\sigma \in [0,t]} \int_0^\sigma \zeta^5_s d\beta_s  \Big] 
& \lesssim
 \E \Big[ \sup_{s \in [0,t]} \Big| Z^{i,N}_s - Z^{j,M}_s \Big|^2 \Big]^{1/2}
  \Big( \int_0^t  \E \Big| Z^{i,N}_s - Z^{j,M}_s \Big|^2 + E [ J_s^{N,M}] ds  \Big)^{1/2}  
\end{align*}
and we may come back \myref{this2} to obtain
\begin{align*}
\E \Big[ \sup_{\sigma \in [0,t]} \Big| Z^{i,N}_\sigma - Z^{j,M}_\sigma \Big|^2 \Big] \lesssim 
| z_0^{i,N} - z_0^{j,M} |^2 +\int_0^t \Big( \E \Big| Z^{i,N}_s - Z^{j,M}_s \Big|^2 + E [ J_s^{N,M}]  \Big) ds.
\end{align*}
Summing over $i$ and $j$ as in \myref{JNM} finally leads to
\begin{align*}
\E \Big[ J^{N,M}_t \Big] \lesssim W_2^2[\mu_0^N, \mu_0^M] + \int_0^t \E [ J^{N,M}_s ] ds.
\end{align*}
Gronwall's lemma hence gives $\E [ J_T^{N,M} ] \lesssim W_2^2[\mu_0^N, \mu_0^M]$ so that, coming back to \myref{JNM1}, we get
\begin{align*}
\E \Big[ W_2^2[\mu^N, \mu^M] \Big] \lesssim W_2^2[\mu_0^N, \mu_0^M] \to 0
\end{align*}
as $N, M$ are sent to infinity. We have shown that $\mu^N$ converges to some $\mu$ in the complete space $L^2(\Omega ; {\cal P}_2({\cal C}))$.
Let us now prove that $\mu$ satisfies the fixed-point identity \myref{fix_point}. First, since $\E \Big[ \int_{z \in {\cal C}} \| z \|_\infty^2 d\mu \Big]~<~\infty$, one could easily deduce from the sub-linearity of the coefficients that
\begin{align*}
\E [ \sup_{t \in [0,T]} \Big| Z^\mu_t(z) \Big|^2 \Big] < \infty,
\end{align*}
thereby guaranteeing that the solution $Z^\mu_t(z)$ of \myref{chap3-chara} is unique and global. Moreover, it is a well known fact (see e.g \cite{kunita}) that the flow $Z^\mu : z \in \mathbb{R}^{2d} \mapsto (Z^\mu_t(z))_{t \in [0,T]}$ is almost-surely continuous, so that the push-forward measure involved in \myref{fix_point} is indeed well-defined. This could be seen in this case by establishing a Kolmogorov estimate
$$
\E \Big[ \sup_{t \in [0,T]} \Big| Z^{\mu}_t(z) - Z^{\mu}_t(z')  \Big|^2 \Big] \lesssim |z-z'|^2, \; \; \forall z,z' \in \mathbb{R}^{2d}.
$$
Let us introduce the measure $\nu = (Z^\mu)^* \mu_0$. Introducing an optimal plan $\pi~\in~\Pi(\mu_0^N, \mu_0)$ so that
$$
W_2^2[\mu_0^N, \mu_0] = \int_{(z_1,z_2) \in (\mathbb{R}^{2d})^2} |z_1-z_2|^2 d\pi(z_1,z_2)
$$
we have this time
$$ 
W_2^2[\mu^N, \nu] \le \int_{(z_1,z_2) \in (\mathbb{R}^{2d})^2}
\sup_{t \in [0,T]} \Big| Z_t^{\mu^N}(z_1) - Z_t^{\mu}(z_2) \Big| ^2 d\pi(z_1,z_2) =: J^N_T.
$$
As in \myref{this1}, Itô's formula gives an expression of the form
\begin{align*}
d  \Big| Z_t^{\mu^N}(z_1) - Z_t^{\mu}(z_2)  \Big|^2 & \lesssim 
\Big( \Big| Z_t^{\mu^N}(z_1) - Z_t^{\mu}(z_2)  \Big|^2 + W_1^2[\mu_t^N, \mu_t] \Big) dt + \zeta_t d\beta_t
\\
& \lesssim 
\Big( \Big| Z_t^{\mu^N}(z_1) - Z_t^{\mu}(z_2)  \Big|^2 + W_2^2[\mu^N, \mu] \Big) dt + \zeta_t d\beta_t.
\end{align*}
Proceeding as in the first part of the proof, we eventually obtain
$$
\E \Big[  W_2^2[\mu^N, \nu ] \Big] \lesssim \E \Big[ W_2^2[\mu^N, \mu] \Big].
$$
Letting $N$ go to infinity, we conclude that $\mu = \nu$ a.s, that is exactly \myref{fix_point}. 

We may once again use the same arguments to prove that the fixed-point-like equation \myref{fix_point} has a unique solution: considering $\mu$ and $\nu$ such that $\mu = (Z^\mu)^* \mu_0$ a.s and $\nu = (Z^\nu)^* \nu_0$ a.s, we are led to 
$$
\mathbb{E} \Big[ W_2^2[\mu,\nu] \Big] \lesssim W_2^2[\mu_0, \nu_0]
$$
so that $\mu_0 = \nu_0$ implies $\mu = \nu$ a.s.

Finally, let us notice that any $\mu$ satisfying \myref{fix_point} defines a solution of \myref{spde_reg}. Indeed, for any test function $\Psi \in C_c^\infty(\mathbb{R}^{2d})$, Itô's formula gives exactly
$$
\Psi(Z^\mu_t(z)) = \Psi(z) + \int_0^t {\cal L}_R[\mu_s] \Psi(Z^\mu_s(z)) ds 
+ \int_0^t \nabla \Psi(Z^\mu_s(z)) \cdot K_R[\mu_s](Z^\mu_s(z)) d\beta_s.
$$
where ${\cal L}_R[\mu] \Psi$ is given by \myref{dual_operator}.
Since $\langle \Psi, \mu_t \rangle = \int_{z \in \mathbb{R}^{2d}} \Psi(Z^\mu_t(z)) d\mu_0(z)$, integrating with respect to $d\mu_0(z)$ using a stochastic Fubini theorem gives the expected result.
\end{proof}

\end{subsection}

\begin{subsection}{Flow of characteristics, regular solutions}

For some fixed $R > 0$, let us consider the unique solution $\mu$ of \myref{spde_reg} constructed in Proposition~\ref{mean_field}. Although this measure, as well as the associated characteristics $(Z^{\mu}_t)_{t \ge 0}$, depend on $R > 0$, we shall hide this dependence in the following expressions to avoid cluttering notation.
One may note from expressions \myref{reg}  that the coefficients 
$$
L^{CS}_R[\mu_t](z), \; \; \; L^{MT}_R[\mu_t](z), \; \; \;  S_R[\mu_t](z), \; \; \;  K_R[\mu_t](z)
$$
involved in the SDE \myref{chap3-chara} have the regularity $C^4(\mathbb{R}^{2d})$ in the $z$ variable. More precisely, assumption \myref{psi} guarantees that, for fixed $R > 0$, for $1 \le |\alpha| \le 4$,
\begin{align*}
\Big| \partial_z^\alpha L^{CS}_R[\mu_t] \Big| +\Big| \partial_z^\alpha  L^{MT}_R[\mu_t] \Big|
+ \Big| \partial_z^\alpha S_R[\mu_t] \Big| + \Big| \partial_z^\alpha K_R[\mu_t] \Big| \lesssim 1
\end{align*}
uniformly in $t \in [0,T]$ and $\omega \in \Omega$. In particular, the first, second and third order $z$-derivatives of the coefficients are globally Lipschitz-continuous.
As a result, Theorem 4.4 of \cite{kunita}, Chapter~II, yields
\begin{align}
\Big[  \; \forall t \in [0,T], \; \;  Z^\mu_t : z \in \mathbb{R}^{2d} \mapsto Z^\mu_t(z) \in \mathbb{R}^{2d}  \text{ is a $C^3$-diffeomorphism} \;  \Big], \; \; \Proba- a.s. 
\label{kunita_dif}
\end{align}

More precisely, for $0 \le s \le t$, denoting by $Z^\mu_{s,t}(z)$ the solution of the SDE
\begin{equation*}
\left\{
\begin{array}{l}
X^\mu_{s,t}(z) = x + \displaystyle{ \int_{r=s}^t V^\mu_{s,r}(z) dr },
\\ \vspace{-3mm} \\
V^\mu_{s,t}(z) = v + \displaystyle{ \int_{r=s}^t \Big( L^{CS}_R[\mu_r]+ L^{MT}_R[\mu_r] + S_R[\mu_r] \Big)(Z^\mu_{s,r}(z)) dr + 
\int_{r=s}^t K_R[\mu_r](Z^{\mu}_{s,r}(z)) d\beta_r }, 
\end{array}
\right.
\end{equation*}
the inverse map $(Z^\mu_{s,t})^{-1}(z)$ satisfies the corresponding backward SDE 
\begin{equation*}
\left\{
\begin{array}{l l}
(X^\mu_{s,t})^{-1}(z) & = x - \displaystyle{ \int_{r=s}^t (V^\mu_{r,t})^{-1}(z) dr },
\\ \vspace{-3mm} \\
(V^\mu_{s,t})^{-1}(z) & = v - \displaystyle{ \int_{r=s}^t \Big( L^{CS}_R[\mu_r]+ L^{MT}_R[\mu_r] + S_R[\mu_r] - \tld  S_R[\mu_r] \Big)((Z ^\mu_{r,t})^{-1}(z)) dr }
\\
& \hspace{50mm} - \displaystyle{
\int_{r=s}^t K_R[\mu_r]((Z^{\mu}_{r,t})^{-1}(z)) \widehat{d\beta_r} }, 
\end{array}
\right.
\end{equation*}
where $\tld S_R[\mu] =  \nabla_v K_R[\mu] K_R[\mu] $ and $\int_{s}^t \cdot \; \widehat{d\beta_t}$ denotes the backward Stratonovich integral (see again \cite{kunita} Theorem 7.3 for this result and p.194 for the definition of the backward integral). When $s=0$, we simply denote 
\begin{align*}
Z^\mu_t(z) = Z^\mu_{0,t}(z), \; \; \; (Z^\mu_t)^{-1}(z) = (Z^\mu_{0,t})^{-1}(z).
\end{align*}
In the particular case where the initial measure $\mu_0$ admits a density $f_0~\in~L^1(\mathbb{R}^{2d})$ with respect to the Lebesgue measure on $\mathbb{R}^{2d}$, for any test function $\Psi \in C_b(\mathbb{R}^{2d})$ we may write
\begin{align*}
\int_{z \in \mathbb{R}^{2d}} \Psi(z) d\mu_t(z)  =\int_{z \in \mathbb{R}^{2d}} \Psi(Z^\mu_t(z)) f_0(z) dz
= \int_{z \in \mathbb{R}^{2d}} \Psi(z) |J_t(z)^{-1} | f_0((Z^\mu_t) ^{-1}(z)) dz 
\end{align*}
where $J_t(z)^{-1}$ denotes the jacobian determinant
\begin{align}
J_t(z)^{-1} =  \det \Big[ D_z ((Z^{\mu}_t)^{-1})(z) \Big].
\label{chap3-J}
\end{align}
Considering a countable separating family of such test functions $\Psi$, from $\mu = (Z^\mu)^* f_0$ in ${\cal P}({\cal C})$, we deduce that
$$
\Big[  \forall t \in [0,T], \; \; d\mu_t(z) = f(t,z) dz \Big], \; \; \Proba-a.s,
$$
where
\begin{align}
f(t,z) = J_t(z)^{-1} f_0((Z^\mu_t) ^{-1}(z)).
\label{density}
\end{align}
Note that we may drop the absolute value in \myref{density} since $J_t(z)^{-1} \ge 0$ a.s.
Let us now give some estimates regarding the forward and backward characteristics.

\begin{proposition} \label{estim_reg}
~~\\
For all $p \ge 1$, let $f_0 \in L^1(\mathbb{R}^{2d})$ such that $ \int |z'|^p f_0(z) ds' < \infty$. 
Let $\mu \in L^2(\Omega ; {\cal P}_2({\cal C}))$ such that $\mu = (Z^\mu)^*f_0$ as in \myref{fix_point}.
Then for all $z,z_1,z_2 \in \mathbb{R}^{2d}$, $t,t_1,t_2 \in [0,T]$ and $1 \le k \le 3$,
\begin{align}
& \E \Big[ \sup_{t \in [0,T]} |Z_t^\mu(z)|^p   \Big] \lesssim 1 + |z|^p + \int |z'|^p f_0(z) dz', \label{bound1}
\\
&   \E \Big[ |D^k_z(Z_t^\mu)(z)|^p   \Big] \lesssim 1 + \int |z'|^p f_0(z) dz', \nonumber
\\
& \E \Big[ \sup_{t \in [0,T]} |Z_t^\mu(z_1) - Z_t^\mu(z_2)|^p \Big] \lesssim |z_1-z_2|^p,  \label{kolmo1}
\\
& \E \Big[ |(Z_t^\mu)^{-1}(z)|^p   \Big] \lesssim |z|^p +  \int |z'|^p f_0(z) dz', \nonumber
\\
&  \E \Big[ \Big| (Z^\mu_{t_1})^{-1})(z_1) -  (Z^\mu_{t_2})^{-1})(z_2)  \Big|^p \Big] \lesssim |z_1-z_2|^p + \Big( 1 + |z_1|^p  +  \int |z'|^p f_0(z) dz' \Big) |t_1-t_2|^{p/2}, \label{kolmo2}
\\
&  \E \Big[ |D^k_z((Z_t^\mu)^{-1})(z)|^p   \Big] \lesssim 1 + \int |z'|^p f_0(z) dz', \label{bound2}
\\
& \E \Big[ \Big| D^k_z ((Z^\mu_{t_1})^{-1})(z) -  D^k_z ((Z^\mu_{t_2})^{-1})(z')  \Big|^p \Big] \lesssim 
|z-z'|^p + |t_1 - t_2|^{p/2}.  \label{kolmo3}
\end{align}
The constants involved in $\lesssim$ depend on $p,k,T,R$ only.
\end{proposition}

These estimates are deduced in a classical manner from the sub-linearity and the global Lipschitz continuity of the coefficients of the equations satisfied by $Z_t^\mu(z)$ and $(Z_t^\mu)^{-1}(z)$.
Applying Kolmogorov's lemma to \myref{kolmo2} and \myref{kolmo3}, we deduce that, for all $0 \le k \le 3$,
\begin{align}
(t,z) \in [0,T] \times \mathbb{R}^{2d} \mapsto D^k((Z^\mu_t)^{-1})(z)
\label{chap3-continuity}
\end{align}
is continuous $\Proba-a.s$. Additionally, we may deduce from \myref{bound1} and \myref{kolmo1} the following estimate: given a compact set $K \subset \mathbb{R}^{2d}$ and $z_0 \in K$, for any $p \ge 1$ and $\alpha \in (0,1)$,
\begin{align}
\E \Big[ \sup_{z \in K} \sup_{t \in [0,T]} | Z_t^\mu(z) |^p \Big] \lesssim \E \Big[ \sup_{t \in [0,T]} | Z_t^\mu(z_0) |^p \Big] + \E \Big[ \| Z^\mu \|_{C^\alpha_z L^\infty_t}^p \Big] \text{diam}(K)^{\alpha p} \le C_{K,T,p} < \infty.
\label{ad_estimate}
\end{align}
In \myref{ad_estimate}, 
$\text{diam}(K) = \sup \left\{| z_1 - z_2 | , \, z_1, z_2 \in K\right\}$ 
and $\| Z^\mu \|_{C^\alpha_z L^\infty_t}$ denotes the $\alpha$-Hölder semi-norm 
$$
\| Z^\mu \|_{C^\alpha_z L^\infty_t} = \sup_{z_1 \neq z_2 \in \mathbb{R}^{2d}} \sup_{t \in [0,T]}   \frac{ | Z^\mu_t(z_1) - Z^\mu_t(z_2) | }{|z_1-z_2|^\alpha}
$$
which satisfies indeed $\E \Big[ \| Z^\mu \|_{C^\alpha_z L^\infty_t}^p \Big] \lesssim 1$ by applying Kolmogorov's lemma to \myref{kolmo1}. We may now establish the following result.

\begin{proposition}[Regular solution] \label{regular}
Let $f_0 \in C_c^2(\mathbb{R}^{2d})$. Let $\mu \in L^2(\Omega ; {\cal P}_2({\cal C}))$ such that 
$\mu~=~(Z^\mu)^* f_0$ as in \myref{fix_point}, and let $f(t,z)$ be defined as in \myref{density}, so that $\mu$ may be represented as $d\mu_t(z) = f(t,z) dz$. Then $f$ is a regular solution of \myref{spde_reg} in the following sense: 
\begin{itemize}
\item
$\Proba$-almost surely, $f(t, \cdot) \in C^2(\mathbb{R}^{2d})$ for all $t \in [0,T]$ and the maps
$$
(t, z) \in [0,T] \times \mathbb{R}^{2d} \mapsto \partial_z^\alpha f(t,z), \; \; 0 \le |\alpha| \le 2
$$
are continuous.

\item Denoting the operator
\begin{align}
({\cal L}_R[f])^* g = - v \cdot \nabla_x g - \nabla_v \cdot \Big( (L_R^{CS}[f] + L_R^{MT}[f] + S_R[f] ) g  \Big)
+ \frac{1}{2} \nabla_v \cdot \nabla_v \Big( K_R[f] K_R[f]^T g \Big),
\label{chap3-operator}
\end{align}
we have for all $z \in \mathbb{R}^{2d}$, $t \in [0,T]$, $\Proba$-almost surely,
\begin{align}
& f(t,z)   =  f_0(z) + \int_0^t ({\cal L}_R[f(s)])^* f(s,z) ds  -\int_0^t  \nabla_v \cdot \Big( K_R[f(s)] f(s,z) \Big) d\beta_s. \label{chap3-strong}
\end{align}
\end{itemize}

\end{proposition}

\begin{proof}

As a consequence of \myref{chap3-continuity}, it is clear from expression \myref{density} that the first condition is met. Furthermore, $f(t, \cdot)$ is almost surely compactly supported, uniformly in $t \in [0,T]$, with
\begin{align*}
\forall t \in [0,T], \; \; \text{Supp}(f(t, \cdot)) \subset \left\{ z \in \mathbb{R}^{2d}, \; \;  |z| \le \sup_{z' \in K} \sup_{t \in [0,T]} | Z^\mu_t(z') | \right\}.
\end{align*}
Since $f$ is a solution of \myref{spde_reg}, for any $\Psi \in C_c^\infty(\mathbb{R}^{2d})$, $\Proba$-almost surely,
$$ \langle f(t) , \Psi \rangle   =  \langle f_0, \Psi \rangle 
+  \int_0^t \Bigl< ({\cal L}_R[f(s)])^* f(s) , \Psi \Bigr> ds
 - \int_0^t \langle \nabla_v \cdot \Big( K_R[f(s)] f(s) \Big) , \Psi \rangle d\beta_s.
$$
We can then interchange the integrals
$$
\int_0^t \langle \cdot , \Psi \rangle ds = \langle \int_0^t  \cdot , \Psi \rangle ds, \hspace{5mm}
\int_0^t \langle \cdot , \Psi \rangle d\beta_s = \langle \int_0^t  \cdot , \Psi \rangle d\beta_s.
$$
Since all functions are compactly supported (for fixed $\omega \in \Omega$), the integrals with respect to $ds$ cause no issue. As for the stochastic integral, we may use a stochastic Fubini theorem, as long as
$$
\int_0^t \int_{z \in \mathbb{R}^{2d}} 
\E \Big[ \Big| \nabla_v \cdot \Big( K_R[f(s)] f(s) \Big) \Psi(z) \Big|^2 \Big] dz ds < \infty.
$$
From expressions \myref{reg}, we deduce (for fixed $R > 0$)
\begin{align*}
& |K_R[f(t)](z)| + | \nabla_v \cdot (K_R[f(t)])(z) | \lesssim 1,
\end{align*}
From expression \myref{density}, (and $|\det[A]| \lesssim |A|^{2d}$) it is clear that
$$
|f(t,z)|^2 + |\nabla_z f(t,z)|^2 \lesssim \Big( 1+  \max_{k=1,2} \Big| D^k_z((Z^\mu_t)^{-1})(z) \Big|^q \Big) \Big(  \| f_0 \|_{L^\infty}^2 + \| \nabla_z f_0 \|_{L^\infty}^2 \Big)
$$
for some $q = q(d) \ge 1$. We may hence write
\begin{align}
 \mathbb{E} \Big| \nabla_v \cdot \Big( K_R[f(s)] f(s) \Big)(z) \Big|^2
& \lesssim
1 + \E \Big[   \max_{k=1,2} \Big| D^k_z((Z^\mu_t)^{-1})(z) \Big|^q   \Big]  \lesssim 1
\label{this_ineq}
\end{align}
thanks to the bound \myref{bound2}. It follows that
$$
\int_0^t \int_{z \in \mathbb{R}^{2d}} 
\E \Big[ \Big| \nabla_v \cdot \Big( K_R[f(s)] f(s) \Big) \Psi(z) \Big|^2 \Big] dz ds 
\lesssim \int_{z \in \mathbb{R}^{2d}} | \Psi(z) |^2 dz < \infty.
$$
Consequently, $\Proba$-a.s, \myref{chap3-strong} holds when integrated against any $\Psi \in C_c^\infty(\mathbb{R}^{2d})$. We deduce that \myref{chap3-strong} holds for almost every $z \in \mathbb{R}^{2d}$. Since both sides of \myref{chap3-strong} are continuous with respect to $z$ (thanks to \myref{this_ineq} again), we conclude that the equality holds for every $z \in \mathbb{R}^{2d}$.

\end{proof}

\end{subsection}

\end{section}

\begin{section}{Weak convergence of approximate solutions} \label{sec3}

From now on, let us fix some initial data $f_0$ satisfying (at least) for some $\delta > 1$,
$$
f_0 \ge 0, \hspace{5mm} \int_{z \in \mathbb{R}^{2d}} f_0(z) dz = 1, \hspace{5mm}  \int_{z \in \mathbb{R}^{2d}} |z|^\delta f_0(z) dz < \infty.
$$
For any $R > 0$, considering the particle system \myref{part_reg}, with initial data satisfying 
$$\mu_0^N(dz) \to f_0(z) dz \text{ in } {\cal P}_\delta(\mathbb{R}^{2d}),$$
we may introduce the solution $\mu_{R}$ of \myref{spde_reg} constructed in Proposition~\ref{mean_field}. As previously discussed, we naturally identify $\mu^R$ with its density ${f^{R}}~=~({f^{R}}(t,z))_{t \in [0,T], z \in \mathbb{R}^{2d}}$ defined by \myref{density}.

\begin{subsection}{Uniform estimates} \label{sec:unif_estimates}

In this section, we shall establish some estimates on  ${f^{R}}$  uniformly on the regularization parameter~$R$.

\begin{proposition} \label{Lp_estimate}
Let $p \ge 1$ and $f_0 \in L^p(\mathbb{R}^{2d})$. Then ${f^{R}} \in L^\infty([0,T] ; L^p(\mathbb{R}^{2d}))$ a.s, with the estimate
$$
\E \Big[ \sup_{t \in [0,T]} \| {f^{R}}(t) \|_{L^p}^p \Big] \lesssim \| f_0 \|_{L^p}^p.
$$
The constant involved in $\lesssim$ depends on $p$ and $T$ only.
\end{proposition}

\begin{proof}
Let us start by considering $f_0 \in C^2_c(\mathbb{R}^{2d})$ supported in some compact $K \subset \mathbb{R}^{2d}$. Then ${f^{R}}$ is a regular solution of \myref{spde_reg} in the sense of Proposition \ref{regular}. For any $z \in \mathbb{R}^{2d}$, applying Itô's formula to $| {f^{R}}(t,z) |^p$ hence gives
\begin{align}
 | {f^{R}}(t,z)|^p & = f_0(z)^p  -p \int_0^t  | {f^{R}}(s,z)|^{p-1} \Big[ v \cdot \nabla_x {f^{R}}(s,z) +
 \nabla_v \cdot \Big( B_{R}[{f^{R}}(s)]  {f^{R}}(s,z)  \Big) \Big] ds
 \nonumber \\ &
 + \frac{p}{2} \int_0^t    | {f^{R}}(s,z)|^{p-1}  \nabla_v \cdot \nabla_v \cdot \Big( K_R[{f^{R}}(s)] K_R[{f^{R}}(s)]^T {f^{R}}(s,z) \Big)  ds
\nonumber  \\
 & + \frac{p(p-1)}{2} \int_0^t | {f^{R}}(s,z)|^{p-2} \Big| \nabla_v \cdot \Big( K_R[{f^{R}}(s)] {f^{R}}(s,z) \Big) \Big|^2 ds
\nonumber  \\
 & - p \int_0^t | {f^{R}}(s,z)|^{p-1} \nabla_v \cdot \Big( K_R[{f^{R}}(s)] {f^{R}}(s,z) \Big) d\beta_s,
 \label{itop}
\end{align}
where $B_{R}[f]$ denotes the drift coefficient $L^{CS}_R[f] + L^{MT}_R[f] + S_R[f]$.
We can now integrate \myref{itop} with respect to $dz$ and interchange the integrals: 
$$
 \int_{z \in \mathbb{R}^{2d}}  \int_0^t \cdot \; \; ds dz = 
  \int_0^t  \int_{z \in \mathbb{R}^{2d}}  \cdot \; \; dz ds, 
  \hspace{5mm}
   \int_{z \in \mathbb{R}^{2d}}  \int_0^t \cdot \; \; d\beta_s dz = 
  \int_0^t  \int_{z \in \mathbb{R}^{2d}}  \cdot \; \; dz d\beta_s,
$$
Since all the integrands in \myref{itop} are compactly supported  in $z$ uniformly for $t \in [0,T]$ (for fixed $\omega \in \Omega$), the integrals with respect to $ds$ cause no issue. As for the stochastic integral, we may use a stochastic Fubnini theorem if we can justify that
\begin{align}
{\cal E} := \E \Big[ \int_0^T \int_{z \in \mathbb{R}^{2d}}
 \Big| {f^{R}}(s,z)^{p-1}  \nabla_v \cdot \Big( K_R[{f^{R}}(s)] {f^{R}}(s,z) \Big)  \Big|^2 dz ds \Big] < \infty.
 \label{fubini}
\end{align}
As in the proof of Proposition \ref{regular}, one can see that (with $\mu(dz) = f^R(z) dz$)
$$
\Big| {f^{R}}(s,z)^{p-1}  \nabla_v \cdot \Big( K_R[{f^{R}}(s)] {f^{R}}(s,z) \Big)  \Big|^2
\lesssim 
1 + \max_{k=1,2}  \Big| D^k_z((Z^\mu_t)^{-1})(z) \Big|^q 
$$
for some $q = q(d,p) \ge 1$. Denoting $S =  \displaystyle{ \sup_{t \in [0,T]} \sup_{z' \in K} |Z_t^\mu(z')| }$, it follows that
\begin{align*}
\int_{z \in \mathbb{R}^{2d}}
 \Big| {f^{R}}(s,z)^{p-1}  \nabla_v \cdot \Big( K_R[{f^{R}}(s)] {f^{R}}(s,z) \Big)  \Big|^2 dz
 \lesssim \int_{|z| \le S} \max_{k=1,2} \Big| D^k_z((Z^\mu_t)^{-1})(z) \Big|^q  dz
\end{align*}
and, fixing some $m > 2d$,
\begin{align*}
&  \int_{|z| \le S}  \max_{k=1,2} \Big| D^k_z((Z^\mu_t)^{-1})(z) \Big|^q  dz
 =  \int_{|z| \le S}  (1+ |z|)^{m/2} \frac{\max_{k=1,2} \Big| D^k_z((Z^\mu_t)^{-1})(z) \Big|^q }{(1+ |z|)^{m/2}}    dz
\\
& \hspace{10mm} \lesssim \int_{|z| \le S} (1+|z|)^m dz + 
\int_{z \in \mathbb{R}^{2d}}  \max_{k=1,2} \Big| D^k_z((Z^\mu_t)^{-1})(z) \Big|^{2q} \frac{dz}{(1+|z|)^m}
\\
& \hspace{10mm} \lesssim (1 + S^{m+1}) + 
\int_{z \in \mathbb{R}^{2d}}  \max_{k=1,2} \Big| D^k_z((Z^\mu_t)^{-1})(z) \Big|^{2q} \frac{dz}{(1+|z|)^m}.
\end{align*}
Using  \myref{ad_estimate} and \myref{bound2}, we deduce
$$
{\cal E}\lesssim 1+ \E \Big[ S^{ m+1} \Big] + \int_{z \in \mathbb{R}^{2d}} 
\E \Big[ \max_{k=1,2} \Big| D^k_z((Z^\mu_t)^{-1})(z) \Big|^{2q} \Big] \frac{dz}{(1+|z|)^m} < \infty
$$
We may hence integrate \myref{itop} with respect to $z$, which leads to
\begin{align}
\int_{\mathbb{R}^{2d}} | {f^{R}}(t,z)|^p dz & = 
\int_{\mathbb{R}^{2d}} f_0(z)^p dz  + I_1(t) + I_2(t) + I_3(t) + I_4(t), \label{chap3-split}
\end{align}
where
\begin{align*}
& I_1(t) =  -p \int_0^t \int_{\mathbb{R}^{2d}}  | {f^{R}}(s,z)|^{p-1} \Big[ v \cdot \nabla_x {f^{R}}(s,z) +
 \nabla_v \cdot \Big( B_{R}[{f^{R}}(s)]  {f^{R}}(s,z)  \Big) \Big] dz ds,
 \\ 
&  I_2(t) =  \frac{p}{2} \int_0^t  \int_{\mathbb{R}^{2d}}   | {f^{R}}(s,z)|^{p-1}  \nabla_v^2 \Big( K_R[{f^{R}}(s)] K_R[{f^{R}}(s)]^T {f^{R}}(s,z) \Big)  dz ds,
 \\
&  I_3(t) =  \frac{p(p-1)}{2} \int_0^t \int_{\mathbb{R}^{2d}} | {f^{R}}(s,z)|^{p-2} \Big| \nabla_v \cdot \Big( K_R[{f^{R}}(s)] {f^{R}}(s,z) \Big) \Big|^2 dz ds, 
 \\
&  I_4(t) =  - p \int_0^t \int_{\mathbb{R}^{2d}} | {f^{R}}(s,z)|^{p-1} \nabla_v \cdot \Big( K_R[{f^{R}}(s)] {f^{R}}(s,z) \Big) dz d\beta_s. 
\end{align*}
Simple calculations lead to the classical identity
\begin{align}
\int |f^R|^{p-1} \nabla_v \cdot ( B_R f^R) dz = \frac{p-1}{p} \int |f^R|^p ( \nabla_v \cdot B_R ) dz
\label{simple_identity}
\end{align}
so that 
\begin{align*}
I_1(t) & =  \frac{p-1}{p} \int_0^t \int_{\mathbb{R}^{2d}} {|f^{R}(s,z)|}^p ( \nabla_v \cdot B_{R} [f(s)](z)) dz ds 
\\
& \le \Big\| \nabla_v \cdot B_{R}[{f^{R}}] \Big\|_{L_{t,z}^\infty([0,T] \times \mathbb{R}^{2d})} \int_0^t \int_{\mathbb{R}^{2d}} |f^{R}(s,z)|^p dz ds.
\end{align*}
From expressions \myref{reg}, we derive
\begin{align*}
& \nabla_v \cdot L^{CS}_R[f](z) =  \int_{\mathbb{R}^{d} \times \mathbb{R}^{d}} \chi_R(x-y) \psi(x-y) \nabla_v \cdot \Big( \theta_R(w-v) \Big) f(y,w) dy dw,
\\
& \nabla_v \cdot L^{MT}_R(z) = -d,
\\
& \nabla_v \cdot K_R[f](z) = \int_{\mathbb{R}^{d} \times \mathbb{R}^{d}} \chi_R(x-y) \tld \psi(x-y) \nabla_v \cdot \Big( \theta_R(w-v) \Big) f(y,w) dy dw,
\\
& \nabla_v \cdot S_R[f](z) = - \frac{1}{2} \Big(  \int_{\mathbb{R}^d \times \mathbb{R}^d}
 \tld \psi(x-y) f(y,w) dy dw \Big) \nabla_v \cdot K_R[f](z)
\end{align*}
Assumptions \myref{psi}, \myref{theta} guarantee that these terms are bounded uniformly in $t \in [0,T]$, $z \in \mathbb{R}^{2d}$, $\omega \in \Omega$, $R > 0$, so that
$$
I_1(t) \lesssim \int_0^t \int_{\mathbb{R}^{2d}} |f^{R}(s,z)|^p dz ds.
$$
Similarly, we have the following identity:
\begin{align*}
& \int |f^R|^{p-1} \nabla_v^2 (K_R K_R^T f^R ) dz + (p-1) \int |f^R|^{p-2} | \nabla_v \cdot (K_R  f^R ) |^2 dz 
\\
& \; \;  = (p-1) \int |f^R|^{p-2} \Big[ | \nabla_v \cdot (K_R f^R)|^2 - \nabla_v f^R \cdot \nabla_v \cdot (K_R K_R^T f^R)  \Big] dz
\\
& \; \; = (p-1) \int |f^R|^p | \nabla_v \cdot K_R |^2
\end{align*}
so that, again,
$$
I_2(t) + I_3(t) = \frac{p(p-1)}{2} \int_0^t \int_{\mathbb{R}^{2d}} |f^{R}(s,z)|^p \Big| \nabla_v \cdot K_R[f^{R}(s)] \Big|^2 dz ds \lesssim  \int_0^t \int_{\mathbb{R}^{2d}} |f^{R}(s,z)|^p  dz ds.
$$
Since \myref{fubini} guarantees that $I_4(t)$ defines a (square integrable) martingale, we may take the expectation in \myref{chap3-split} and apply Gronwall's lemma to derive
\begin{align*}
\mathbb{E} \Big[ \int_{\mathbb{R}^{2d}} | {f^{R}}(t,z) |^p dz  \Big] \lesssim \| f_0 \|_{L^p}^p.
\end{align*}
It follows that
$$
\E \Big[ \sup_{\sigma \in [0,t]} \int_{\mathbb{R}^{2d}} | {f^{R}}(\sigma,z) |^p dz \Big] \lesssim \int_0^t 
\E \Big[ \int_{\mathbb{R}^{2d}}
|{f^{R}}(s,z) |^p dz \Big] ds + \E \Big[ \sup_{\sigma \in [0,t]} I_4(\sigma) \Big]
$$
and Burkholder-Davis-Gundy's inequality yields (making use of \myref{simple_identity} again)
\begin{align*}
\E \Big[  \sup_{\sigma \in [0,t]} I_4(\sigma) \Big] & \lesssim 
 \E\Big[ \Big(  \int_0^t \Big| \int_{\mathbb{R}^{2d}} | {f^{R}} |^{p-1} \nabla_v \cdot (K_R[{f^{R}}] {f^{R}}) dz \Big|^2 d\sigma \Big)^{1/2} \Big]
 \\
 & \lesssim 
  \E\Big[ \Big(  \int_0^t \Big| \int_{\mathbb{R}^{2d}} | {f^{R}}(\sigma,z) |^{p} dz \Big|^2 d\sigma \Big)^{1/2} \Big]
\\
& \le \frac{1}{2} \E \Big[  \sup_{\sigma \in [0,t]} \int_{\mathbb{R}^{2d}} | {f^{R}}(\sigma,z) |^p dz \Big] 
+ C \int_0^t \E \Big[  \int_{\mathbb{R}^{2d}} | {f^{R}}(\sigma,z) |^{p} dz \Big]  d\sigma 
\end{align*}
which gives the expected result.
We now extend the estimate to any $f_0 \in L^p(\mathbb{R}^{2d})$ satisfying $\int |z|^\delta f_0(z) dz < \infty$  by considering a sequence of densities $f_0^k~\in~C_c^2(\mathbb{R}^{2d})$ such that, as $k$ goes to infinity, 
\begin{align*}
& f_0^k \to f_0 \text{ a.s and in } L^p(\mathbb{R}^{2d}),
\\
& \sup_{k \ge 1} \int_{z \in \mathbb{R}^{2d}} |z|^\delta f_k(z) dz < \infty.
\end{align*}
It is easy to see that these assumptions imply in particular
$$
f_0^k \to f_0 \text{ in } {\cal P}_r(\mathbb{R}^{2d})
$$
for any $1 < r < \delta$. Denoting by $f^k$ the solution of \myref{spde_reg} with initial data $f_0^k$ constructed in Proposition \ref{mean_field}, we may deduce (as in the proof of Proposition \ref{mean_field}), 
\begin{align*}
\E \Big[ \sup_{t \in [0,T]} W_r^r[ f^k_t, f_t ]  \Big] \le \E \Big[ W_r^r[ f^k, f ]  \Big] \lesssim W_r^r[f_0^k, f_0]
\to 0.
\end{align*}
Up to a subsequence, we may hence assume that 
\begin{align}
\sup_{t \in [0,T]} W_r[f_t^k, f_t] \to 0, \; \Proba - a.s.
\label{W_as}
\end{align}
From the estimates
$$
\E \Big[ \sup_{t \in [0,T]} \| f^k(t) \|_{L^p}^p \Big] \lesssim \| f_0^k \|_{L^p}^p, \; k \ge 1,
$$
we derive that $(f^k)_{k \ge 1}$ is bounded in $L^p_\omega L^\infty_t L^p_z$ and therefore, up to a subsequence
\begin{align}
f^k \rightharpoonup g \text{ weak $*$ in } L^p\Big(\Omega ;  L^\infty([0,T] ;  L^p(\mathbb{R}^{2d}) ) \Big)
\label{this_weak}
\end{align}
where $g$ satisfies the bound
\begin{align}
\E \Big[ \sup_{t \in [0,T]} \| g(t) \|_{L^p}^p \Big] \lesssim \limsup_k \| f_0^k \|_{L^p}^p = \| f_0 \|_{L^p}^p.
\label{bound_g}
\end{align}
Let us consider $\Psi \in C_c([0,T] \times \mathbb{R}^{2d})$, $\xi \in L^\infty(\Omega)$ and introduce $\Phi(\omega,t,z) = \xi(\omega) \Psi(t,z)$. 
From \myref{W_as} we deduce
\begin{align*}
\Big[ \forall t \in [0,T], \; \int_z f^k(t,z) \Psi(t,z) dz \to \int_z f(t,z) \Psi(t,z) dz \Big], \; \Proba - a.s.
\end{align*}
and the bound 
$\Big| \int_z f^k(t,z) \Psi(t,z) dz  \Big| \le \| \Psi \|_{L^\infty_{t,z}}$ guarantees
\begin{align*}
\int_{t=0}^T \int_z f^k(t,z) \Psi(t,z) dz dt  \to \int_{t=0}^T \int_z f(t,z) \Psi(t,z) dz dt, \; \Proba - a.s.
\end{align*}
Finally, the bound 
$
\Big| \int_{t=0}^T \int_z \xi(\omega) f^k(\omega, t,z) \Psi(t,z) dz dt \Big|
\le T \| \xi \|_{L^\infty_\omega} \| \Psi \|_{L^\infty_{t,z}}
$
guarantees
\begin{align*}
\E \Big[ \xi(\omega) \int_{t=0}^T \int_z f^k(\omega, t,z) \Psi(t,z) dz dt  \Big]  \to & \E \Big[ \xi(\omega) \int_{t=0}^T \int_z f(\omega, t,z) \Psi(t,z) dz dt \Big]
\end{align*}
so that, according to \myref{this_weak},
\begin{align*}
\E \Big[ \xi \int_{t=0}^T \int_z g(t,z) \Psi(t,z) dz dt \Big]
= 
\E \Big[ \xi \int_{t=0}^T  \int_z f(t,z) \Psi(t,z) dz dt \Big].
\end{align*}
We easily derive that $f=g$ in $L^p_\omega L^\infty_t L^p_z$ and the bound \myref{bound_g} concludes the proof.
\end{proof}

\begin{proposition} \label{ineq_u}
For all $f : \mathbb{R}^{2d} \to \mathbb{R}^+$ and $p \ge 1$,
\begin{align*}
\int_{z \in \mathbb{R}^{2d}} |u_{R}[f](x)|^p f(z) dz 
\lesssim \int_{z \in \mathbb{R}^{2d}}  |v|^p f(z) dz.
\end{align*}
The constant involved in $\lesssim$ depends on $\phi$ only.

\end{proposition}

\begin{proof}
It is clear from the expression of $u_R[f]$ \myref{reg} and Jensen's inequality that
\begin{align*}
\int_{z \in \mathbb{R}^{2d}} |u_{R}[f](x)|^p f(z) dz
& \le \int_{x} \int_{v} \frac{ \int_{y} \int_{w} |w|^p \phi(x-y) f(y,w) dy dw }{ \int_y \int_w \phi(x-y) f(y,w) dy dw} f(x,v) dx dv 
\\
&=  \int_y \int_w  \Big( \int_x \frac{\rho(x)}{\tld \rho(x)} \phi(x-y) dx \Big) |w|^p f(y,w) dx dy dw
\end{align*}
where
\begin{align*}
\rho(x) = \int_{v \in \mathbb{R}^d} f(x,v), \hspace{5mm} \tld \rho(x) = (\phi * \rho)(x) = \int_{y \in \mathbb{R}^d} \int_{w \in \mathbb{R}^d} \phi(x-y) f(y,w) dy dw.
\end{align*}
The desired estimate hence follows from the inequality
\begin{align}
\forall y \in \mathbb{R}^d, \; \; \int_{x \in \mathbb{R}^d} \frac{\rho(x)}{\tld \rho(x)} \phi(x-y) dx \le C(\phi)
\label{ineq_kmt}
\end{align}
where, with assumption \myref{phi} in mind, $C(\phi)$ is some constant proportional to 
\begin{align}
C(\phi) \propto \frac{\sup_{B(0,r_2)} \phi}{\inf_{B(0,r_1)} \phi} (R/r)^d.
\label{C_phi}
\end{align}
The proof of \myref{ineq_kmt} is given in \cite{KMT1}, Lemma 5.2.

\end{proof}

\begin{proposition} \label{chap3-moments}
Let $k \ge 2$, $1 < \delta \le 2$ and $f_0$ be a density satisfying $\int_z (|x|^\delta + |v|^{k}) f_0(z) dz < \infty$. Then,
\begin{align*}
& \E \Big[ \sup_{t \in [0,T]} \int_{\mathbb{R}^{2d}} |v|^{k} {f^{R}}(t,z) dz  \Big] \lesssim  1+
\int_{\mathbb{R}^{2d}} |v|^{k} f_0(z) dz.
\\
& \E \Big[ \sup_{t \in [0,T]} \int_{\mathbb{R}^{2d}} |x|^{\delta} {f^{R}}(t,z) dz  \Big] \lesssim  1+
\int_{\mathbb{R}^{2d}} (|x|^\delta + |v|^{2}) f_0(z) dz.
\end{align*}
The constants involved in $\lesssim$ depends on $k$, $\delta$, $T$ and $\phi$ only.
\end{proposition}

\begin{proof}
The first estimate should first be established with the stopping time 
$$ \tau_M = \inf \left\{ t \ge 0, \; \; \int_{z \in \mathbb{R}^{2d}} |v|^{k} f_t(z) dz \ge M \right\} \wedge T,$$
which should then be sent to $T$, as $M$ goes to infinity. For the sake of simplicity, we omit this stopping time in the following.
Let us denote $f(t,z) = {f^{R}}(t,z)$ and the associated characteristics $Z_t(z)~=~Z_t^{{\mu^{R}}}(z)$ (with $\mu^R(dz) = f^R(z) dz$), satisfying \myref{chap3-chara}. We have
$$
 \int_{z \in \mathbb{R}^{2d}} |z|^{k} f(t,z) dz = \int_{z \in \mathbb{R}^{2d}} \Big|  Z_t(z) \Big|^{k} f_0(z) dz
$$
and Itô's formula gives, for fixed $z \in \mathbb{R}^{2d}$,
\begin{align*}
d \Big[ |V_t(z) |^{k} \Big] =&  k | V_t(z) |^{k - 2} V_t(z) \cdot \Big(  L^{CS}_R[f_t] + L^{MT}_R[f_t] + S_R[f_t]   \Big)(Z_t(z)) dt 
\\
& + k | V_t(z) |^{k-2} V_t(z) \cdot K_R[f_t](Z_t(z)) d\beta_t
\\
& + k(k/2-1) |V_t|^{k-4} \Big| V_t(z) \cdot K_R[f_t](Z_t(z)) \Big|^2 dt + k |V_t(z) |^{k-2} \Big| K_R[f_t](Z_t(z)) \Big|^2 dt.
\end{align*}
From the uniform sublinearity \myref{sublinear}, we derive
\begin{align}
 |V_t(z) |^{k} = |v|^{k} + \int_0^t \Big( a^1_s(z) + a^2_s(z)  \Big) ds + \int_0^t a^3_s(z) d\beta_s,
 \label{this_sde}
\end{align}
where
\begin{align*}
& | a^1_t(z)| + |a^3_t(z)| \lesssim \Big( 1 + |V_t(z)|^{k} + \int_{z'} |v'|^{k} f_t(z') dz' \Big),
\\
& a^2_t(z) = k |V_t(z)|^{k-2} V_t(z) \cdot \Big( u_{R}[f_t](X_t(z)) - V_t(z) \Big).
\end{align*}
We may then integrate \myref{this_sde} with respect to $f_0(z) dz$ using a stochastic Fubini theorem, which leads to
\begin{align}
 \int_{z \in \mathbb{R}^{2d}} |v|^{k} f(t,z) dz = \int_{z \in \mathbb{R}^{2d}} |v |^{k} f_0(z) dz + 
 \int_0^t \Big( A_s^1 + A_s^2 \Big) ds + \int_0^t A_s^3 d\beta_s,
 \label{this_sde2}
\end{align}
where
\begin{align}
& | A^1_t| + |A^3_t| \lesssim \Big( 1 + \int_{z} |v|^{k} f_t(z) dz \Big), \label{sublin_A}
\\
& A^2_t = k \int_{z \in \mathbb{R}^{2d}} |v|^{k-2} v \cdot \Big( u_{R}[f_t](x) - v \Big) f_t(z) dz. \nonumber
\end{align}
To deal with $A^2_t$, one may write
$$
v \cdot (u-v) = \frac{1}{4} \Big( |v + (u-v) |^2 - |v-(u-v)|^2  \Big) \le \frac{1}{4} |u|^2
$$
so that 
\begin{align*}
\int_z |v|^{k-2} v \cdot (u(x) - v) f(z) dz & \le \frac{1}{4} \int_z |v|^{k-2} |u(x)|^2 f(z) dz
\\ &
 \le  \frac{1}{4} \int_z |v|^{k} f(z) dz +  \frac{1}{4}  \int_z |u(x)|^{k} f(z) dz.
\end{align*}
From Proposition \ref{ineq_u}, we hence deduce
\begin{align}
| A^2_t | \lesssim \int |v|^{k} f_t(z) dz.
\label{sublin_A2}
\end{align}
From SDE \myref{this_sde2} with the sublinear terms \myref{sublin_A} and \myref{sublin_A2}, using Gronwall's lemma and Burkholder-Davis-Gundy's inequality, we classically obtain the first estimate.
Moreover, from
\begin{align*}
| X_t(z) |^{\delta} = \Big| x + \int_0^t V_s(z) ds \Big|^{\delta} \lesssim 1+ |x|^{\delta} + \int_0^t |V_s(z)|^{2} ds
\end{align*}
since $\delta \le 2$, we derive the second estimate, which concludes the proof.
\end{proof}

\end{subsection}

\begin{subsection}{Stochastic averaging lemma} \label{sec:av_lemma}

\begin{proposition} \label{averaging_lemma}
Let us assume that the initial data $f_0$ satisfies, for some $\theta \in (0,1)$,
\begin{align}
\int_{z \in \mathbb{R}^{2d}} |f_0(z)|^{p} dz + \int_{z \in \mathbb{R}^{2d}} (1 + |v|^k) f_0(z) dz <  \infty \; \text{ with }  \; p = 1 + \frac{1}{\theta}, \hspace{4mm} k > \frac{4}{1-\theta}.
\label{compromise1}
\end{align}
For all, $\varphi \in C_c^\infty(\mathbb{R}^d)$, letting $\eta = 1/6$, the averaged quantity
\begin{align*}
\rho_{\varphi}^R(t,x) = \int_{\mathbb{R}^d} \varphi(v) f^R(t,x,v) dx
\end{align*}
lies in $L^2([0,T] ; H^\eta(\mathbb{R}^d))$ almost surely, with the estimate
\begin{align}
\E \Big[ \| {\rho_\varphi^{R}} \|_{L^2_t H^\eta_x}^2   \Big] 
 \lesssim 1.
\label{averaging}
\end{align}
The constant involved in $\lesssim$ in \myref{averaging} depends on $f_0$, $\varphi$, $T$ and $\phi$ only.

\end{proposition}
\begin{proof}
The proof of this result is based on a classical $L^2$ averaging lemma (see e.g \cite{bouchut} for the deterministic case), which we adapt here to the stochastic case, in a similar fashion to \cite{debussche}, Lemma 4.3.
Let us first consider some initial data $f_0 \in C_c^2(\mathbb{R}^{2d})$ so that $f = {f^{R}}$ is a regular solution of \myref{spde_reg} in the sense of Proposition \ref{regular}, which may be written as follows: $\Proba$-a.s, for all $z \in \mathbb{R}^{2d}$,
\begin{align}
d f(t,z) + v \cdot \nabla_x f(t,z) dt = & \Big( \sum_{1 \le i \le d} \partial_{v_i} G^i(t,z) + \sum_{1 \le i,j \le d} \partial^2_{v_i v_j} G^{i j}(t,z) \Big) dt 
\nonumber \\ &
+ \Big( \sum_{1 \le i \le d} \partial_{v_i} H^i(t,z) \Big) d\beta_t
\label{regular_sde}
\end{align} 
where we have introduced the coefficients
\begin{equation}
\left\{
\begin{array}{l}
G^i(t,z) = \Big(L^{CS}_R[f(t)] + L^{MT}_R[f(t)] + S_R[f(t)] \Big)(z)^i f(t,z),
\vspace{1mm}
\\
G^{ij}(t,z) = K_R[f(t)](z)^i  K_R[f(t)](z)^j f(t,z),
\vspace{1mm}
\\
H^i(t,z) = K_R[f(t)]^i f(t,z).
\end{array}
\right.
\label{GH}
\end{equation}
Let us fix $\xi \in \mathbb{R}^d$. For simplicity, let us drop the summation signs in \myref{regular_sde} and integrate it with respect to $e^{-i \xi \cdot x} dx$. This is possible (for every $v \in \mathbb{R}^d$) thanks to the bound \myref{fubini} with $p=1$ established previously. Denoting the $x$-Fourier transform
$$
\hat f(t,\xi,v) = \int_{x \in \mathbb{R}^d} e^{-i \xi \cdot x} f(t,x,v) dx,
$$
we are led to
\begin{align*}
d \hat f(t,\xi,v) + i v \cdot \xi \hat f(t,\xi,v) dt = 
+ \Big( \widehat{\partial_{v_i}G^i} + \widehat{\partial^2_{v_i v_j }G^{ij}} \Big)(t,\xi,v) dt.
+ \widehat{ \partial_{v_i} H^i}(t,\xi,v) d\beta_t
\end{align*}
Therefore, introducing some $\lambda = \lambda(\xi) > 0$, we get
\begin{align*}
d \hat f(t,\xi,v) + (\lambda + i v \cdot \xi) \hat f(t,\xi,v) dt =  \lambda \hat f(t,\xi,v) dt 
& + \Big( \widehat{\partial_{v_i}G^i} + \widehat{\partial^2_{v_i v_j }G^{ij}} \Big)(t,\xi,v) dt
\nonumber \\ &
+ \widehat{ \partial_{v_i} H^i}(t,\xi,v) d\beta_t,
\end{align*}
from which we deduce the expression
\begin{align}
\hat f(t,\xi,v) = & \;  e^{-(\lambda + i v \cdot \xi)t} \hat{f_0}(\xi,v)
+ \lambda \int_0^t  e^{-(\lambda + i v \cdot \xi)(t-s)} \hat f(s,\xi,v) ds
\nonumber \\ &
+ \int_0^t  e^{-(\lambda + i v \cdot \xi)(t-s)}  \Big( \widehat{\partial_{v_i}G^i} + \widehat{\partial^2_{v_i v_j }G^{ij}} \Big)(s,\xi,v)ds
\nonumber \\ &
+  \int_0^t  e^{-(\lambda + i v \cdot \xi)(t-s)}  \widehat{ \partial_{v_i} H^i}(s,\xi,v) d\beta_s.
\label{sde_lambda2}
\end{align}
We now integrate \myref{sde_lambda2} with respect to $\varphi(v) dv$. This is possible since one could show that
\begin{align*}
& \E \int_0^t \int_{v \in \mathbb{R}^d} \Big| \varphi(v) e^{-(\lambda + i v \cdot \xi)(t-s)}  \widehat{ \partial_{v_i} H^i}(s,\xi,v)  \Big|^2 dv ds
\\
& \hspace{5mm} \lesssim \E \int_0^T \int_{v \in \mathbb{R}^d} (1 + |v|^2) \Big|( \int_{x \in \mathbb{R}^d} \Big| \nabla_v \cdot (K_R[f(s)] f(s)) \Big| dx  \Big)^2 dv ds < \infty
\end{align*}
for fixed $R > 0$, by a method similar  to the one employed to establish \myref{fubini} in the proof Proposition \ref{Lp_estimate}. Introducing the $(x,v)$-Fourier transform
$$
({\cal F}f) (t,\xi, \zeta) = \int_{v \in \mathbb{R}^d} \int_{x \in \mathbb{R}^d} e^{-i (\xi \cdot x + \zeta \cdot v)} f(t,x,v) dx dv,
$$
equation \myref{sde_lambda2} leads to
\begin{align*}
\widehat{\rho_\varphi}(t,\xi) = & e^{-\lambda t} {\cal F}(f_0 \varphi)(\xi, \xi t) +
 \lambda \int_0^t e^{-\lambda(t-s)} {\cal F}(f \varphi)(s,\xi, \xi(t-s)) ds
 \\ &
 + \int_0^t e^{-\lambda(t-s)} \Big( {\cal F}((\partial_{v_i} G^i) \varphi) + {\cal F}((\partial^2_{v_i v_j} G^{i j}) \varphi)  \Big)(s, \xi, \xi(t-s)) ds
 \\ &
 + \int_0^t e^{-\lambda(t-s)} {\cal F}((\partial_{v_i} H^i) \varphi)(s, \xi, \xi(t-s)) d\beta_s.
\end{align*}
Note that since 
\begin{align}
& (\partial_{v_i} G^i) \varphi = \partial_{v_i}( G^i \varphi) - G^i \partial_{v_i} \varphi,
\label{identity_diff_1}
\\
& (\partial_{v_i v_j}^2 G^{ij}) \varphi = \partial_{v_i v_j}^2 (G^{ij} \varphi) - \partial_{v_i}(G^{ij} \partial_j \varphi) -  \partial_{v_j}(G^{ij} \partial_i \varphi) + G^{ij} \partial^2_{v_i v_j} \varphi
\label{identity_diff_2}
\end{align}
forgetting the summation signs again, we may as well only consider terms of the form
\begin{align}
\widehat{\rho_\varphi}(t,\xi) = & e^{-\lambda t} {\cal F}(f_0 \Psi)(\xi, \xi t) +
 \lambda \int_0^t e^{-\lambda(t-s)} {\cal F}(f \Psi)(s,\xi, \xi(t-s)) ds
 \nonumber \\ &
 + \int_0^t e^{-\lambda(t-s)} {\cal F}\Big( \partial_v^\alpha(G^\beta \Psi)  \Big)(s, \xi, \xi(t-s)) ds
  \nonumber \\ &
 + \int_0^t e^{-\lambda(t-s)} {\cal F} \Big( \partial_v^\gamma(H^i \Psi) \Big)(s, \xi, \xi(t-s)) d\beta_s
 \label{spde_rho}
\end{align}
where $\Psi \equiv \Psi(v) \in C_c^\infty(\mathbb{R}^d)$ and $\alpha, \beta$ and $\gamma$ are multi-indexes of order $0 \le |\alpha| \le 2$, $1 \le |\beta| \le 2$ and $0 \le |\gamma| \le 1$. From \myref{spde_rho} Itô's isometry gives
\begin{align}
\E \int_0^T | \widehat{\rho_\varphi}(t,\xi) |^2 dt \lesssim J_1(\xi) + J_2(\xi) + J_3(\xi) + J_4(\xi) 
\label{rho_isometry}
\end{align}
where
\begin{align*}
& J_1(\xi) = \int_0^T e^{-2\lambda t} | {\cal F}(f_0 \Psi)(\xi, \xi t) |^2 dt, 
\\
& J_2(\xi) = \lambda^2 \E \int_0^T \Big| \int_0^t e^{-\lambda(t-s)} {\cal F}(f \Psi)(s,\xi, \xi(t-s)) ds  \Big|^2 dt,
 \\
& J_3(\xi) = \E \int_0^T \Big|  \int_0^t e^{-\lambda(t-s)} {\cal F}\Big( \partial_v^\alpha(G^\beta \Psi)  \Big)(s, \xi, \xi(t-s)) ds  \Big|^2 dt,
\\
& J_4(\xi) = \E \int_0^T \int_0^t e^{-2 \lambda(t-s)}\Big| {\cal F} \Big( \partial_v^\gamma(H^i \Psi) \Big)(s, \xi, \xi(t-s)) \Big|^2 ds \, dt.
\end{align*}
Using a trace lemma (see e.g \cite{triebel}, Theorem 2.7.2)  we get, for some $s > (d-1)/2$,
\begin{align*}
J_1(\xi) & \le  \int_{\mathbb{R}}  \Big| {\cal F}(f_0 \Psi)(\xi,\xi t) \Big|^2 dt
= |\xi|^{-1} \int_{\mathbb{R}}  \Big| {\cal F}(f_0 \Psi)(\xi,\frac{\xi}{|\xi|} t) \Big|^2 dt
\\ &
\lesssim |\xi|^{-1} \int_{\mathbb{R}^d} \Big| (Id - \Delta_w)^s {\cal F}(f_0 \Psi)(\xi, w)  \Big|^2 dw.
\end{align*}
It then follows by Plancherel's identity, since $\Psi$ is compactly supported, that
\begin{align}
J_1(\xi) \lesssim  |\xi|^{-1} \int_{\mathbb{R}^d} (1 + |v|^2)^s | \Psi(v)|^2 | \hat f_0(\xi,v) |^2 dv
\lesssim  |\xi|^{-1}  \int_{\mathbb{R}^d}  | \hat f_0(\xi,v) |^2 dv.
\label{bound_J1}
\end{align}
For the second term $J_2(\xi)$, Jensen's inequality gives
\begin{align*}
J_2(\xi) & \le \lambda \E \int_0^T \int_0^t e^{-\lambda(t-s)} \Big| {\cal F}(f \Psi)(s,\xi,\xi(t-s)) \Big|^2 ds dt
\\
& \le  \lambda \E \int_0^T \int_s^T \Big| {\cal F}(f \Psi)(s,\xi,\xi(t-s)) \Big|^2 dt ds 
\le  \lambda \E \int_0^T \Big( \int_{\mathbb{R}} \Big| {\cal F}(f \Psi)(s,\xi,\xi t ) \Big|^2 dt  \Big) ds 
\end{align*}
and the same manipulation leads to
\begin{align}
J_2(\xi) \lesssim \lambda  |\xi|^{-1} \E \int_0^T \int_{\mathbb{R}^d} |\hat f(s, \xi, v) |^2 dv ds.
\label{bound_J2}
\end{align}
Similarly, for the third term $J_3(\xi)$, Jensen's inequality gives
\begin{align*}
J_3(\xi) & \le \lambda^{-1} \E \int_0^T \int_0^t e^{-\lambda(t-s)} \Big| {\cal F}(\partial^\alpha_v(G^\beta \Psi))(s,\xi,\xi(t-s)) \Big|^2 ds dt
\\
& \le \lambda^{-1}  \E \int_0^T \int_0^t e^{-\lambda(t-s)} |\xi|^4 (t-s)^4 \Big| {\cal F}(G^\beta \Psi)(s,\xi,\xi(t-s)) \Big|^2 ds dt
\\
& \lesssim \lambda^{-5} |\xi|^4  \E \int_0^T \Big( \int_{\mathbb{R}}  \Big| {\cal F}(G^\beta \Psi)(s,\xi,\xi t) \Big|^2 dt \Big) ds
\end{align*}
where we have used $e^{-\lambda(t-s)} (t-s)^4 \lesssim \lambda^{-4}$. 
The same manipulation hence leads to
\begin{align}
J_3(\xi) \lesssim \lambda^{-5} |\xi|^3 \E \int_0^T \int_{\mathbb{R}^d} 
| \widehat{G^\beta}(s,\xi,v) |^2 dv ds.
\label{bound_J3}
\end{align}
Finally, for the fourth term $J_4(\xi)$, we get
\begin{align*}
J_4(\xi) & \le \E \int_0^T \int_0^t e^{-2 \lambda(t-s)} |\xi|^2 (t-s)^2 \Big| {\cal F}(H^i \Psi)(s,\xi,\xi(t-s)) \Big|^2 ds dt
\\ & \lesssim \lambda^{-2} |\xi|^2 \E \int_0^T \Big( \int_{\mathbb{R}} \Big| {\cal F}(H^i \Psi)(s,\xi,\xi(t-s)) \Big|^2 dt \Big) ds
\end{align*}
and the same manipulation leads to
\begin{align}
J_4(\xi) \lesssim \lambda^{-2} |\xi| \E \int_0^T \int_{\mathbb{R}^d} | \widehat{H^i}(s,\xi,v) |^2 dv ds .
\label{bound_J4}
\end{align}
Now fixing some $\eta > 0$, let us consider
\begin{align*}
\E \int_0^T \int_{\mathbb{R}^d} (1 + |\xi|^{2\eta}) | \widehat{\rho_\varphi}(t,\xi) |^2 d\xi dt
= I_1 + I_2
\end{align*}
where
\begin{align}
& I_1 = \E \int_0^T \int_{\mathbb{R}^d}  | \widehat{\rho_\varphi}(t,\xi) |^2 d\xi dt +  \E \int_0^T \int_{|\xi | \le 1} |\xi|^{2\eta} | \widehat{\rho_\varphi}(t,\xi) |^2 d\xi dt,
\label{I1}
\\
& I_2 = \E \int_0^T \int_{|\xi | > 1} |\xi|^{2\eta} | \widehat{\rho_\varphi}(t,\xi) |^2 d\xi dt.
\label{I2}
\end{align}
On one hand, making use of Plancherel's identity, and since $\varphi$ is compactly supported,
\begin{align*}
I_1  \le 2
 \E \int_0^T \int_{\mathbb{R}^d} | \widehat{\rho_\varphi}(t,\xi) |^2 d\xi dt
 & \lesssim  \E \int_0^T \int_{\mathbb{R}^d} |\rho_\varphi(t,x) |^2 dx dt
 \\ &
  \lesssim \E \int_0^T \int_{\mathbb{R}^d} \int_{\mathbb{R}^d} |f(t,x,v)|^2 dv dx dt \lesssim 1
\end{align*}
thanks to Proposition \ref{Lp_estimate}, and the initial bound \myref{compromise1}.
On the other hand, from the bounds \myref{bound_J1}, \myref{bound_J2}, \myref{bound_J3} and \myref{bound_J4} we deduce (using Plancherel's identity once again)
\begin{align}
I_2 \lesssim & \sup_{|\xi > 1} \Big[  |\xi|^{2 \eta - 1} + \lambda(\xi) |\xi|^{2 \eta - 1} + \lambda(\xi)^{-5} |\xi|^{2 \eta + 3} + \lambda(\xi)^{-2} |\xi|^{2 \eta + 1}  \Big]
\nonumber
\\
& \times \Big[
\int_{z \in \mathbb{R}^{2d}}  |f_0|^2 dz
+ \E \int_0^T \int_{z \in \mathbb{R}^{2d}}  \Big(|f|^2 + |G^\beta|^2 + |H^i|^2 \Big) dz dt
\Big].
\label{second_part}
\end{align}
Expressions \myref{GH} and the sublinearity estimate \myref{sublinear} immediately give
\begin{align}
& |G^i(t,z)|^2 \lesssim \Big( 1 + |v|^2 + \int |v'|^2 f(t,z') dz' + \Big| u_{R}[f(t)](x) \Big|^2 \Big) |f(t,z)|^2,
\label{ineq_G1}
\\
& |G^{ij}(t,z)|^2 \lesssim \Big( 1 + |v|^4 + \int |v'|^4 f(t,z') dz' \Big) |f(t,z)|^2,
\label{ineq_G2}
\\
& |H^i(t,z)|^2 \lesssim  \Big( 1 + |v|^2 + \int |v'|^2 f(t,z') dz' \Big) |f(t,z)|^2,
\label{ineq_G3}
\end{align}
from which we easily deduce (using the same method as in the proof of Proposition \ref{chap3-moments} for the term involving $u_{R}[f(t)]$), for $\theta \in (0,1)$,
\begin{align*}
& \int_{z \in \mathbb{R}^{2d}} |G^i|^2 dz \lesssim \int_{z \in \mathbb{R}^{2d}} (1 + |v|^{\frac{2}{1-\theta}}) f(t) dz + \int_{z \in \mathbb{R}^{2d}}  |f(t)|^{1 + \frac{1}{\theta}} dz,
\\
& \int_{z \in \mathbb{R}^{2d}} |G^{ij}|^2 dz \lesssim \int_{z \in \mathbb{R}^{2d}} (1 + |v|^{\frac{4}{1-\theta}}) f(t) dz + \int_{z \in \mathbb{R}^{2d}}  |f(t)|^{1 + \frac{1}{\theta}} dz,
\\
& \int_{z \in \mathbb{R}^{2d}} |H^i|^2 dz \lesssim \int_{z \in \mathbb{R}^{2d}} (1 + |v|^{\frac{2}{1-\theta}}) f(t) dz + \int_{z \in \mathbb{R}^{2d}}  |f(t)|^{1 + \frac{1}{\theta}} dz,
\end{align*}
so that, thanks again to Proposition \ref{Lp_estimate}, Proposition \ref{chap3-moments} and the initial bound \myref{compromise1}, \myref{second_part} yields
\begin{align*}
I_2 \lesssim \sup_{|\xi > 1} \Big[  |\xi|^{2 \eta - 1} + \lambda(\xi) |\xi|^{2 \eta - 1} + \lambda(\xi)^{-5} |\xi|^{2 \eta + 3} + \lambda(\xi)^{-2} |\xi|^{2 \eta + 1}  \Big].
\end{align*}
Considering $\lambda(\xi) = |\xi|^r$, this supremum is bounded under the requirements
\begin{align*}
2 \eta - 1 \le 0, \hspace{5mm} 2 \eta - 1 + r \le 0,  \hspace{5mm} 2 \eta + 3 - 5 r \le 0,  \hspace{5mm} 2 \eta + 1 - 2 r \le 0,
\end{align*}
which can be met as soon as $\eta \le 1/6$. For such $\eta$, we have shown that
\begin{align*}
\E \Big[ \| \rho_\varphi \|_{L^2_t H^\eta_x}^2  \Big] \lesssim \E \int_0^T \int_{\mathbb{R}^d} (1 +| \xi |^{2\eta})  | \widehat{\rho_\varphi}(t,\xi) |^2 d\xi dt \lesssim 1.
\end{align*}
which concludes the proof in the case of a regular initial data $f_0 \in C_c^2(\mathbb{R}^{2d})$. We may extend this last inequality to general initial data similarly to the proof of Proposition~\ref{Lp_estimate}.
\end{proof}

\end{subsection}

\begin{subsection}{Tightness}

Given an increasing weight function, say
\begin{align*}
W(x,v) = 1 + |x| + |v|,
\end{align*}
let us introduce the weighted Sobolev space
\begin{align}
H^2_{W}(\mathbb{R}^{2d}) = \left\{ \Psi \in D'(\mathbb{R}^{2d}), \; \| \Psi \|_{H^2_{W}}^2 := \max_{|\beta| \le 2} \int |\partial_z^\beta \Psi |^2(z) W(z) dz < \infty \right\}
\label{weight1}
\end{align}
and the dual space
\begin{align}
H^{-2}_{W^{-1}}(\mathbb{R}^{2d}) = \Big(  H^2_{W^{-1}}(\mathbb{R}^{2d}) \Big)' \text{ with }
\| h \|_{H^{-2}_{W^{-1}}} = \sup \left\{ \langle h, \Psi \rangle, \; \| \Psi \|_{H^2_{W}} = 1 \right\}.
\label{weight2}
\end{align}
We may also define the intermediate Sobolev spaces $H^\sigma_W(\mathbb{R}^{2d})$,  $H^{-\sigma}_W(\mathbb{R}^{2d})$ for non-integer $0 < \sigma < 2$.
Note that $\| \Psi \|_{H^2} \le \| \Psi \|_{H^2_W}$ so that $\| h \|_{H^{-2}_{W^{-1}}} \le \| h \|_{H^{-2}}$.

\begin{proposition} \label{tight1}
Let us assume that the initial data $f_0$ satisfies, for some $\delta > 1$ and $\theta \in (0,1)$,
\begin{align}
\int_{z \in \mathbb{R}^{2d}} |f_0(z)|^{p} dz + \int_{z \in \mathbb{R}^{2d}} ( |x|^\delta + |v|^k) f_0(z) dz <  \infty \; \text{ with }  \; p = 1 + \frac{1}{\theta}, \hspace{4mm} k > \frac{4}{1-\theta}.
\label{compromise2}
\end{align}
Then for all $\sigma > 0$, the family of random variables $(f^R)_{R > 0}$ is tight in 
$C([0,T] ; H^{-\sigma}_{W^{-1}}(\mathbb{R}^{2d}))$ .
\end{proposition}

\begin{proof}
Without loss of generality, let $\sigma \in (0, 2)$. For some $\alpha \in (0,1/2)$ and $M > 0$, let us introduce the set
$$
K_M = \left\{ f \in C([0,T] ; H^{-\sigma}_{W^{-1}}(\mathbb{R}^{2d})) \; \Big| \; \| f \|_{L^\infty_t L^2_z} \le M,  \; \| f \|_{C^\alpha_t H_{W^{-1}}^{-2}} \le M \right\},
$$
where $\| \cdot \|_{C^\alpha_t H^{-2}_W}$ denotes the $\alpha$-Hölder semi-norm
\begin{align*}
\| f \|_{C^\alpha_t H^{-2}_{W^{-1}}} := \sup_{t \neq s} \frac{ \| ft) - f(s) \|_{H^{-2}_{W^{-1}}}} {|t-s|^\alpha}
\le \| f \|_{C^\alpha_t H^{-2}}.
\end{align*}
Since the embedding
$H^\sigma_{W}(\mathbb{R}^{2d})  \subset L^2(\mathbb{R}^{2d})$ is compact,  the dual embedding $L^2(\mathbb{R}^{2d})  \subset H^{-\sigma}_{W^{-1}}(\mathbb{R}^{2d})$ is compact. Additionally, for $f \in K_M$, an interpolation inequality (in weighted Sobolev spaces) yields, for some $\tau \in (0,1)$,
 $$
 \| f \|_{C^\alpha_t H^{-\sigma}_{W^{-1}}} \lesssim 
\| f \|_{L^\infty_t L^2_{W^{-1}}}^\tau  \| f \|_{C^\alpha_t H^{-2}_{W^{-1}}}^{1-\tau}
\lesssim M^2.
$$
Consequently,  Arzelà-Ascoli's theorem guarantees that $K_M$ is a relatively compact subset of the separable, complete space
 $C([0,T] ; H^{-\sigma}_{W^{-1}}(\mathbb{R}^{2d}))$.
Markov's inequality gives
$$
\Proba \Big[ f^R \notin K_M \Big] \le M^{-2} \E \Big[ \sup_{t \in [0,T]} \| f^R(t) \|^2_{L^2}  \Big] + M^{-\gamma} \E \Big[  \| f^R \|^\gamma_{C^\alpha_t H^{-2}} \Big] \Big).
$$
for $\gamma > 0$.
The first term is bounded uniformly in $R$ thanks to Proposition \ref{Lp_estimate}. The bound on the second term results directly from Kolmogorov's continuity theorem and the following lemma.

\begin{lemma} \label{kolmo_f}
For some $q > 1$, for all $t,s \in [0,T]$,
\begin{align*}
\E \Big[ \| f^R(t) - f^R(s) \|_{H^{-2}}^{2q}  \Big] \lesssim |t-s|^{q}.
\end{align*}
The constant involved in $\lesssim$ depends on $q$, $f_0$ and $T$ only.
\end{lemma}
We now prove this lemma: let $0 \le s \le t \le T$. Let us simply denote $f = f^R$ and note that
\begin{align*}
\E \Big[ \| f(t) - f(s) \|_{H^{-2}}^{2q}  \Big] =
 \E \Big[ \Big|  \int_{\mathbb{R}^{d}} \int_{\mathbb{R}^d} (1 + |\xi|^4 + |\zeta|^4)^{-1} \Big| {\cal F}f(t,\xi,\zeta) - {\cal F}f(s,\xi,\zeta)  \Big|^2 d\xi d\zeta \Big|^q \Big].
\end{align*}
As in the proof of Proposition \myref{averaging_lemma}, we may apply the 
$(x,v)$-Fourier transform ${\cal F} f(t,\xi,\zeta)$ to equation \myref{regular_sde} to get (forgetting the summation signs)
\begin{align*}
{\cal F}f(t) - {\cal F}f(s) =  \int_s^t \Big(  -i \xi \cdot {\cal F}(v f) + i \zeta^i {\cal F}(G^i) 
-\zeta^i \zeta^j {\cal F}(G^{ij}) \Big) d\sigma + i \int_s^t \zeta^i {\cal F}(H^i) d\beta_\sigma.
\end{align*}
Itô's formula results in
\begin{align*}
 \Big| {\cal F}f(t) - {\cal F}f(s)  \Big|^2  = & 2 \int_s^t   \overline{ {\cal F} f (\sigma) - {\cal F} t(s) }  \Big(  -i \xi \cdot {\cal F}(v f) + i \zeta^i {\cal F}(G^i) 
-\zeta^i \zeta^j {\cal F}(G^{ij}) \Big) d\sigma
\\
& + \int_s^t | \zeta^i |^2 | {\cal F}(H^i) |^2 d\sigma 
+ 2 i \int_s^t \overline{ {\cal F} f (\sigma) - {\cal F} t(s) } \zeta^i {\cal F}(H^i) d \beta_\sigma
\end{align*}
so that, integrating against $(1 + |\xi|^4 + |\zeta|^4)^{-1} d\xi d\zeta$, we get
\begin{align}
\| f(t) - f(s) \|^2_{H^{-2}} \lesssim & \int_s^t \| f(\sigma) - f(s) \|^2_{H^{-2}} d \sigma + D_t + M_t
\label{D_M}
\end{align}
where
\begin{align*}
& D_t = 
\int_s^t \int_{\mathbb{R}^{2d}} \Big( |{\cal F}(vf)|^2 + |{\cal F}(G^i)|^2 + |{\cal F}(G^{ij})|^2 + |{\cal F}(H^i)|^2 \Big) d\xi d\zeta d\sigma,
\\
& M_t = 2i \int_s^t \Big( \int_{\mathbb{R}^{2d}} (1 + |\xi|^4 + |\zeta|^4)^{-1} \overline{ {\cal F} f (\sigma) - {\cal F} t(s) } \zeta^i {\cal F}(H^i) d\xi d\zeta \Big) d\beta_\sigma.
\end{align*}
From \myref{D_M} we derive
\begin{align}
\E \| f(t) - f(s) \|^{2q}_{H^{-2}} \lesssim  \int_s^t \E \| f(\sigma) - f(s) \|^{2q}_{H^{-2}} d \sigma + \E \Big[ | D_t |^{q}  \Big] + \E \Big[ | M_t |^{q} \Big].
\label{D_M2}
\end{align}
Plancherel's identity gives
\begin{align*}
|D_t|^{q} & = \Big| \int_s^t \Big( \| v f \|_{L^2}^2 + \| G^i \|^2_{L^2} + \| G^{ij} \|^2_{L^2} + \| H^i \|^2_{L^2} \Big) d\sigma \Big|^{q} 
\\
& \lesssim |t-s|^{q-1} 
\int_s^t  \Big( \| v f \|_{L^2}^{2q} + \| G^i \|^{2q}_{L^2} + \| G^{ij} \|^{2q}_{L^2} + \| H^i \|^{2q}_{L^2} \Big) d\sigma
\end{align*}
so that
\begin{align*}
\E \Big[ |D_t|^q \Big] \lesssim |t-s|^q \E \Big[ \sup_{t \in [0,T]} 
\Big( 
 \| v f(t) \|_{L^2}^{2q} + \| G^i(t) \|^{2q}_{L^2} + \| G^{ij}(t) \|^{2q}_{L^2} + \| H^i(t) \|^{2q}_{L^2}
  \Big)   \Big]
\end{align*}
Considering \myref{ineq_G1}, \myref{ineq_G2}, \myref{ineq_G3}, we see that we essentially need a bound (uniform in $R$) on
\begin{align}
\E \Big[ \sup_{t \in [0,T]} \Big| \int_{\mathbb{R}^{2d}} (1 + |v|^4) |f(t,z)|^2 dz \Big|^q \Big].
\label{bound_needed}
\end{align}
This is possible since, for any $m > 2d$,
\begin{align*}
\Big| \int_{\mathbb{R}^{2d}} (1 + |v|^4) |f(t,z)|^2 dz \Big|^q & \lesssim 
\int_{\mathbb{R}^{2d}} (1 + |v|^{4q}) (1 + |z|)^{m(q-1)} |f|^{2q} dz
\\ &
\lesssim \int_{\mathbb{R}^{2d}} (1 + |v|^{4q+m(q-1)} + |x|^{m(q-1)} ) |f|^{2q} dz
\end{align*}
so that, for $\tau \in (0,1)$,
\begin{align*}
\Big| \int_{\mathbb{R}^{2d}} (1 + |v|^4) |f(t,z)|^2 dz \Big|^q & \lesssim 
\int_{\mathbb{R}^{2d}} (1 + |x|^{\frac{m(q-1)}{1-\tau}} + |v|^{\frac{4q+m(q-1)}{1-\tau}}  ) f dz
+ \int_{\mathbb{R}^{2d}} |f|^{1 + \frac{2q-1}{\tau}} dz
\\
& \lesssim \int_{\mathbb{R}^{2d}} (1 + |x|^\delta + |v|^k) f dz + \int_{\mathbb{R}^{2d}} |f|^p dz
\end{align*}
whenever, recalling \myref{compromise2}, for some $\gamma > 0$,
\begin{align*}
\frac{2q-1}{\tau} = \frac{1}{\theta}, \hspace{10mm} \frac{m(q-1)}{1-\tau} \le \delta, \hspace{10mm}
\frac{4q + m(q-1)}{1- \tau} \le k := \frac{4 + \gamma}{1-\theta}.
\end{align*}
These requirements can be met for some $q=q(\gamma) > 1$ close enough to $1$ 
(and $\tau$ close to $\theta$). As for the martingale term, Burkholder-Davis-Gundy's inequality gives
\begin{align}
\E \Big[ |M_t|^q \Big] & \lesssim \E \Big[ \Big| \int_s^t \Big( \int_{\mathbb{R}^{2d}} (1 + |\xi|^4 + |\zeta|^4)^{-1} | {\cal F} f (\sigma) - {\cal F} t(s) |  | \zeta |  | {\cal F}(H^i) | d\xi d\zeta \Big)^2 d\sigma \Big|^{q/2}  \Big]
\nonumber \\
& \lesssim \E \Big[ \Big| \int_s^t  \| f(\sigma) - f(s) \|_{H^{-2}}^2 \| H^i \|_{L^2}^2 d\sigma \Big|^{q/2} \Big]
\nonumber \\
&
\lesssim \E \Big[ \Big| \sup_{t \in [0,T]} \| H^i(t) \|^2_{L^2} \Big|^{q/2} \Big|   \int_s^t  \| f(\sigma) - f(s) \|_{H^{-2}}^2 d\sigma \Big|^{q/2} \Big]
\nonumber  \\
&
\lesssim \E \Big[ \sup_{t \in [0,T]} \| H^i(t) \|^{2q}_{L^2}  \Big] +   \int_s^t  \E \| f(\sigma) - f(s) \|_{H^{-2}}^{2q} d\sigma.
\label{H_term}
\end{align}
Considering \myref{ineq_G3}, the first term in \myref{H_term} is again controlled by the bound on \myref{bound_needed}. We may then come back to \myref{D_M2} and use Grönwall's lemma to conclude.

\end{proof}

\begin{proposition} \label{tight2}
Let us assume that the initial data $f_0$ satisfies \myref{compromise2}. 
\\
Then, for all $\varphi \in C_c^\infty(\mathbb{R}^d)$, the family of random variables $(\rho^{R}_\varphi)_{R > 0}$ is tight in 
$L^2([0,T] ; L^2(\mathbb{R}^d))$.

\end{proposition}

\begin{proof}
Let us introduce the weight function
$$
V(x) = 1 + |x|
$$
and the associated weighted spaces $H^2_{V}(\mathbb{R}^d)$ and $H^{-2}_{V^{-1}}(\mathbb{R}^d)$ as in \myref{weight1} and \myref{weight2}. 
One could prove the following lemma as previously.
\begin{lemma} \label{kolmo_rho}
For some $q > 1$, for all $t,s \in [0,T]$,
\begin{align*}
\E \Big[ \| \rho^{R}_\varphi(t) - \rho^{R}_\varphi(s)  \|_{H^{-2}}^{2q}  \Big]  \lesssim |t-s|^q.
\end{align*}
The constant involved in $\lesssim$ depends on $q$, $f_0$, $\varphi$ and $T$ only.\end{lemma}
We may now fix some $M > 0$ and naturally introduce the set
\begin{align*}
K_M = \left\{\rho_\varphi := \int \varphi(v) f dv \; | \; f \in {\cal F}_M \right\},
\end{align*}
where ${\cal F}_M$ denotes the set of functions $f \equiv f(t,x,v)$ satisfying, for $\eta = 1/6$ and some $\alpha \in (0,1/2)$,
\begin{align}
& \sup_{t \in [0,T]} \| f(t) \|_{L^p_z}^p \le M, \label{Lp_M}
\\
& \sup_{t \in [0,T]}  \int ( 1 + |x|^\delta + |v|^k) f(t) dz \le M, \label{moment_M}
\\
& \| \rho_\varphi \|_{L^2_t H^\eta_x}^2 \le M, \label{average_M}
\\
& \| \rho_\varphi \|_{C^\alpha_t H^{-2}} \le M. \nonumber
\end{align}
Markov's inequality, gives, for some $\gamma > 0$,
\begin{align*}
\Proba \Big[ \rho^{R}_\varphi \notin K_M \Big] & \le 
M^{-1} \E \Big[  \sup_{t \in [0,T]} \| f^R(t) \|_{L^p_z}^p   \Big]
+ M^{-1} \E \Big[  \sup_{t \in [0,T]}  \int ( 1 + |x|^\delta + |v|^k) f^R(t) dz \Big]
\\ & \; \;
+ M^{-1} \E \Big[ \| {\rho^{R}_\varphi} \|_{L^2_t H^\eta_x}^2 \Big]
+ M^{-\gamma} \E \Big[ \| {\rho^{R}_\varphi} \|_{C^\alpha_t H^{-2}}^\gamma \Big]
\end{align*}
which tends to zero uniformly in $R > 0$ as $M$ goes to infinity, thanks to Proposition \ref{Lp_estimate}, \ref{chap3-moments}, \ref{averaging_lemma} and Lemma \ref{kolmo_rho}. It only remains to prove that $K_M$ is a relatively compact subset of 
$ L^{2}([0,T] ; L^{2}(\mathbb{R}^d)) $. 
Let us introduce a sequence $(\rho_\varphi^n)_{n \ge 1}$ in $K_M$.
\vspace{3mm}

First, let us show that $(\rho_\varphi^n)_n$ is compact locally in space, that is  in 
$L^{2} \Big( [0,T] ; L^{2}(B(0,r)) \Big)$ for any $r > 0$. Since 
$ \| \rho^n_\varphi \|_{L^\infty_t L^1_x} \le \| \varphi \|_{L^\infty}$, we deduce from \myref{Lp_M} that
$$
\| \rho^n_\varphi \|_{L^\infty_t L^2_x} \lesssim M. 
$$
Similarly to the proof of Proposition \ref{tight1}, we may then use Arzelà-Ascoli's theorem to deduce that  $(\rho_\varphi^n)_{n}$ converges in $C([0,T] ; H^{-2}_{V^{-1}}(\mathbb{R}^d))$ up to some subsequence (which we omit for clarity). An interpolation inequality (in weighted Sobolev spaces: $H^\eta_{V^{-1}} \subset L^2_{V^{-1}} \subset H^{-2}_{V^{-1}}$) yields
\begin{align*}
\| \rho \|_{L^2(B(0,r))} \lesssim \| \rho \|_{L^2_{V^{-1}}} \lesssim \| \rho \|_{H^{-2}_{V^{-1}}}^\tau 
\| \rho \|_{H^\eta_{V^{-1}}}^{1-\tau}
\lesssim \| \rho \|_{H^{-2}_{V^{-1}}}^\tau 
\| \rho \|_{H^\eta}^{1-\tau}
\end{align*}
where $\tau = \frac{\eta}{\eta + 2} \in (0,1)$. It follows that
\begin{align*}
\int_0^T \| \rho_\varphi^n(t) - \rho_\varphi^m(t) \|_{L^2(B(0,r))}^{2} dt 
& \le \int_0^T \| \rho_\varphi^n(t) - \rho_\varphi^m(t) \|_{H^{-2}_{V^{-1}}}^{2 \tau } \times  
\| \rho_\varphi^n(t) - \rho_\varphi^m(t) \|_{H^{\eta}}^{2 (1-\tau) }
dt 
\\ & 
\le \Big( \int_0^T  \| \rho_\varphi^n(t) - \rho_\varphi^m(t) \|_{H^{-2}_{V^{-1}}}^{2 } dt  \Big)^\tau
\Big( \int_0^T  \| \rho_\varphi^n(t) - \rho_\varphi^m(t) \|_{H^{\eta}}^{2 } dt  \Big)^{1-\tau}
\\ &
\le T^\tau \| \rho_\varphi^n - \rho_\varphi^m \|_{C([0,T] ; H^{-2}_{V^{-1}})}^{2 \lambda} \times 
\| \rho_\varphi^n - \rho_\varphi^m \|^{2 (1-\tau)}_{L^{ 2 }_t H^\eta}.
\end{align*}
Thanks to \myref{average_M}, we deduce that
\begin{align*}
\| \rho_\varphi^n - \rho_\varphi^m \|_{L^{2}([0,T] ;  L^2(B(0,r))} \to 0 \; \text{ as } n,m \to \infty.
\end{align*}
To derive compactness globally in space, that is in $L^{2}([0,T] ; L^{2}(\mathbb{R}^d))$, it is enough to establish a uniform integrability estimate of the form
$$
\sup_{n \ge 1} \Big[ \int_0^T  \int_{|x| \ge r} |\rho_\varphi^n(t,x)|^{2} dx  dt \Big] \to 0 \; \text{ as } r \to \infty.
$$
To this intent, since $\varphi$ is compactly supported, we may simply write
\begin{align*}
\int_{|x| \ge r} | \rho^n_\varphi(x) |^2 dx \lesssim \int_{\mathbb{R}^d} \int_{|x| \ge r} |f^n(x,v)|^2 dx dv
& \le r^{-\gamma} \int_{\mathbb{R}^d} \int_{\mathbb{R}^d} (1 + |x|^\gamma) |f^n(x,v)|^2 dx dv
\\
& \lesssim r^{-\gamma} \Big( \int_{\mathbb{R}^{2d}} (1 + |x|^\delta) f^n(z) dz + \int_{\mathbb{R}^{2d}} |f^n(z)|^p dz \Big)
\end{align*}
where $\gamma = \delta(1-\theta) > 0$, according to \myref{compromise2}, and  then use the bounds \myref{Lp_M} and \myref{moment_M} to conclude.
\end{proof}

\end{subsection}

\begin{subsection}{Convergence of the martingale problem}
Let us introduce a sequence $R_n \to \infty$ and  a countable subset ${\cal D}$ of $C^\infty_c(\mathbb{R}^d)$, which we assume to contain the truncation functions
\begin{align}
\left\{ \theta_{R_n}, n \ge 1  \right\} \subset {\cal D} \subset C_c^\infty(\mathbb{R}^d).
\label{calD}
\end{align}
Recall that the function $\theta_R(v)$ has been introduced in \myref{theta}. Since ${\cal D}$ is countable, it follows from Proposition~\ref{tight1}, Proposition \ref{tight2} and Tykhonov's theorem that the family of random variables 
$(f^{R_n}, (\rho^{R_n}_\varphi)_{\varphi \in {\cal D}})_{n \ge 1}$ is tight in the space
$$C([0,T] ; H^{-\sigma}_{W^{-1}}(\mathbb{R}^{2d})) \times \Big( L^2([0,T] ; L^2(\mathbb{R}^d) \Big)^{\cal D}$$ for $\sigma > 0$.
Using Skorokhod's representation theorem, up to a subsequence of $(R_n)_n$ which we omit for simplicity, we may introduce random variables $\myhat f^n, \myhat \rho_\varphi^n, \myhat f, \myhat \rho_\varphi$ defined on some other probability space $(\myhat \Omega, \myhat{\cal F}, \myhat \Proba)$ such that, for all $n \ge 1$,
\begin{align*}
\hspace{-3mm}
 (\myhat f^n,  (\myhat \rho^n_\varphi)_{\varphi \in {\cal D}}) \sim 
 (f^{R_n},  ( \rho^{R_n}_\varphi)_{\varphi \in {\cal D}})
 \text{ in law, in }
C([0,T] ; H^{-\sigma}_{W^{-1}}(\mathbb{R}^{2d})) \times \Big( L^2([0,T] ; L^2(\mathbb{R}^d))\Big)^{\cal D},
\end{align*}
and  the following convergences hold $\myhat \Proba$-almost surely:
\begin{align*}
& \myhat f^n \to \myhat f \text{ in } C([0,T] ; H^{-\sigma}_{W^{-1}}(\mathbb{R}^{2d})) \text{ a.s}, 
\\
& \forall \varphi \in {\cal D}, \; \;  \myhat \rho^n_\varphi \to \myhat \rho_{\varphi} \text{ in } L^2([0,T] ; L^2(\mathbb{R}^d)) \text{ a.s}.
\end{align*}
More precisely, since Proposition \ref{Lp_estimate} and Proposition \ref{averaging_lemma} provide the bounds
\begin{align}
\myhat \E \Big[ \| \myhat f^n \|_{L^\infty_t L^2_x}^2 \Big] =  \E \Big[ \|  f^n \|_{L^\infty_t L^2_x}^2 \Big] \lesssim 1, \hspace{5mm}
\myhat \E \Big[ \| \myhat \rho^n_\varphi \|^2_{L^2_t H^\eta_x} \Big] = \E \Big[ \|  \rho^n_\varphi \|^2_{L^2_t H^\eta_x} \Big] \lesssim 1
\label{uniform_e}
\end{align}
uniformly in $n \ge 1$, we derive in particular that the families of random variables $(\myhat f^n)_n$ and $(\myhat \rho^n_\varphi)_n$ are uniformly integrable, and therefore
\begin{align}
& \myhat f^n \to \myhat f \text{ a.s and in } L^1\Big(\myhat \Omega ; C([0,T] ; H^{-\sigma}_{W^{-1}}(\mathbb{R}^{2d})) \Big).
\label{skoro3}
\\
& \forall \varphi \in {\cal D}, \; \;  \myhat \rho^n_\varphi \to \myhat \rho_{\varphi} \text{ a.s and in } L^1\Big(\myhat \Omega ; L^2([0,T] ; L^2(\mathbb{R}^d)) \Big).
\label{skoro2}
\end{align}

Consequently, up to a subsequence, we may also assume that
\begin{align}
\forall \varphi \in {\cal D}, \; \; \Big[ \myhat \rho^n_\varphi(t,x) \to \myhat \rho_\varphi (t,x), \; dt \otimes dx\text{-a.e} \Big], \; \; \myhat \Proba\text{-a.s}.
\label{ae_as_1}
\end{align}

\begin{remark}
For all $\varphi \in {\cal D}$, from the equality
$\rho^n_\varphi = \int \varphi(v) f^n dv$,  $\Proba$-a.s, it is clear that
$\myhat \rho^n_\varphi = \int_{\mathbb{R}^d} \varphi(v) \myhat f^n dv$,  $\myhat \Proba$-a.s.
The convergence \myref{skoro3} then guarantees that, for all $\varphi \in {\cal D}$, $\myhat \rho_\varphi$ is indeed given by
$$
\myhat \rho_\varphi = \int_{\mathbb{R}^d} \varphi(v) \myhat f dv,  \; \text{ $\myhat \Proba$-a.s}.
$$
\end{remark}
Let us now introduce the averaged quantities
\begin{align*}
\rho = \int_{\mathbb{R}^d} f dv, \hspace{10mm} j = \int_{\mathbb{R}^d} v f dv.
\end{align*}
By requiring some greater moments for the initial data, we may extend the convergence \myref{skoro2} to $\rho$ and $j$.

\begin{lemma} \label{rho_j}
Assume that the initial data $f_0$ satisfies, for some $\theta \in (0,1)$,
\begin{align}
\int_{z \in \mathbb{R}^{2d}} |f_0(z)|^{p} dz + \int_{z \in \mathbb{R}^{2d}} (1 + |v|^k) f_0(z) dz <  \infty \; \text{ with }  \; p = 1 + \frac{1}{\theta}, \hspace{4mm} k > \frac{d + 2}{1-\theta}.
\label{compromise3}
\end{align}
Then the following convergences hold in $L^1\Big(\myhat \Omega ; L^2([0,T] ; L^2(\mathbb{R}^d)) \Big)$:
\begin{align*}
& \myhat \rho^n \to \myhat \rho, \hspace{10mm} \myhat j^n \to \myhat j, \hspace{10mm}
\myhat \rho^n_{\theta_{R_n}} \to \myhat j,
\\
& \phi * \myhat \rho^n \to \phi * \myhat \rho, \hspace{10mm} \phi * \myhat j^n \to \phi * \myhat j,
 \hspace{10mm} \phi * \myhat \rho^n_{\theta_{R_n}} \to \phi * \myhat j
\end{align*}
Consequently, up to a subsequence, we may also assume that these convergences hold
$dt \otimes dx$ almost everywhere, $\myhat \Proba$ almost surely.
\end{lemma}

\begin{proof} Let us, for instance, prove the convergence of $\myhat j^n$.
For fixed $N \ge 1$, denoting $\varphi = \theta_{R_N}$,
\begin{align}
\| \myhat j^n - \myhat j \|_{L^2_t L^2_x} \le \| \myhat j^n -  \myhat \rho^n_{\varphi} \|_{L^2_t L^2_x}
+ \| \myhat \rho^n_\varphi - \myhat \rho_\varphi  \|_{L^2_t L^2_x} + \| \myhat \rho_\varphi - \myhat j \|_{L^2_t L^2_x}
\label{split1}
\end{align}
and we note that, for any $\gamma > 0$ and $m > d$,
\begin{align*}
| \myhat j^n -  \myhat \rho^n_{\varphi}|^2 & = \Big| \int |v- \theta_{R_N}(v)| \myhat f^n dv \Big|^2
\lesssim \Big| \int_{|v| \ge R_N} |v| \myhat f^n dv \Big|^2 \le R_N^{-2\gamma} \Big| \int_{\mathbb{R}^d} |v|^{1 + \gamma} \myhat f^n dv  \Big|^2
\\
& \lesssim R_N^{-2\gamma} \int_{\mathbb{R}^d} (1 + |v|)^{2 + 2 \gamma + m} |\myhat f^n|^2 dv
\\
&
\lesssim R_N^{-2\gamma} \Big( \int_{\mathbb{R}^d} (1 + |v|)^{\frac{2 + 2\gamma + m}{1- \theta}} \myhat f^n dv 
+ \int_{\mathbb{R}^d} |\myhat f^n|^{1 + \frac{1}{\theta}} dv  \Big)
\\
& \lesssim R_N^{-2\gamma} \Big( \int_{\mathbb{R}^d} (1 + |v|)^{k} \myhat f^n dv 
+ \int_{\mathbb{R}^d} |\myhat f^n|^{p} dv  \Big)
\end{align*}
for $(\gamma,m)$ close enough to $(0,d)$. As a result,
\begin{align*}
\sup_{n \ge 1} \myhat \E \Big[  \| \myhat j^n -  \myhat \rho^n_{\varphi} \|_{L^2_t L^2_x}^2  \Big] \lesssim  R_N^{-2 \gamma} \to 0 \text{ as } N \to \infty.
\end{align*}
Since, from \myref{skoro3}, we classically derive
$$
\myhat \E \Big[ \int_0^T \int_{\mathbb{R}^{2d}} \Big(  (1 + |v|)^{k} \myhat f + |\myhat f|^p \Big) dz dt
 \Big] 
 \le \sup_{n \ge 1} \myhat \E \Big[ \int_0^T \int_{\mathbb{R}^{2d}} \Big(  (1 + |v|)^{k} \myhat f^n + |\myhat f^n|^p \Big) dz dt
 \Big]  \lesssim 1
$$
we deduce similarly that $ \myhat \E \Big[  \| \myhat j -  \myhat \rho_{\varphi} \|_{L^2_t L^2_x}^2  \Big] \to 0$ as $N$ goes to infinity. We may hence come back to \myref{split1} and conclude. The convergence of the convoluted functions is easily deduced.

\end{proof}

For fixed $n \ge 1$, $f^{R_n}$ defines a (strong) solution of \myref{spde_reg} on $(\Omega, {\cal F}, \Proba)$. In particular, it satisfies the associated martingale problem: recalling the operator ${\cal L}_R[f]$ defined in \myref{dual_operator},
for all $\Psi \in C_c^\infty(\mathbb{R}^{2d})$, the process
\begin{align}
& M^n_\Psi(t) = \langle \Psi, f^{R_n}(t) \rangle -  \langle \Psi, f_0 \rangle  - \int_0^t \Bigl< {\cal L}_{R_n}[f^{R_n}(s)] \Psi, f^{R_n}(s) \Bigr> ds, \hspace{5mm} t \in [0,T]
\label{chap3-martingale}
\end{align}
defines a continuous, real valued $L^2$ martingale on $(\Omega, {\cal F}, \Proba)$ with respect to the filtration
$$
{\cal F}^n_t = \sigma \Big( f^{R_n}(s) \in H^{-\sigma}_{W^{-1}}(\mathbb{R}^{2d}), \; \; s \in [0,t]  \Big), \; \; \; t \in [0,T].
$$
Its quadratic variation is given by
\begin{align}
\Big[ M^n_\Psi \Big](t) = V^n_\Psi(t) :=  \int_0^t \Big| \Bigl< K_{R_n}[f^{R_n}(s)] \cdot \nabla_v \Psi, f^{R_n}(s) \Bigr>  \Big|^2 ds.
\label{quad_var}
\end{align}

We are now ready to state the following result.

\begin{proposition} \label{martingale_problem}
Let us introduce, on $(\myhat \Omega, \myhat {\cal F}, \myhat \Proba)$, the filtration
$$
\myhat {\cal F}_t = \sigma \Big( \myhat f(s) \in H^{-\sigma}_{W^{-1}}(\mathbb{R}^{2d}), \; \; s \in [0,t]  \Big), \; \; \; t \in [0,T].
$$
Recalling the operator ${\cal L}[f]$ defined in \myref{limit_operator}, for all test functions $\Psi$ of the form
\begin{align}
\Psi(x,v) = \Psi_1(x) \Psi_2(v), \hspace{5mm} \Psi_1, \Psi_2 \in C_c^\infty(\mathbb{R}^d) \text{ with } \nabla_v \Psi_2 \in {\cal D},
\label{form_psi}
\end{align}
the process
\begin{align*}
& \myhat M_\Psi(t) = \langle \Psi, \myhat f(t) \rangle -  \langle \Psi, f_0 \rangle  - \int_0^t \Bigl< {\cal L}[\myhat f(s)] \Psi, \myhat f(s) \Bigr> ds, \hspace{5mm} t \in [0,T]
\end{align*}
defines a continuous, real-valued $L^2$ martingale with respect to $(\myhat {\cal F}_t)_{t \ge 0}$, with quadratic variation
\begin{align*}
\Big[ \myhat M_\Psi \Big](t) = \myhat V_\Psi(t) := \int_0^t \Big| \Bigl< K[ \myhat f(s)] \cdot \nabla_v \Psi, \myhat f(s) \Bigr>  \Big|^2 ds.
\end{align*}

\end{proposition}
\begin{remark} \label{rem_psi}
The assumption $\nabla_v \Psi_2 \in {\cal D}$ in \myref{form_psi} is only technical: for any given countable family $F~=~(\Psi_2)_{\Psi_2 \in F}$ of test functions in $C_c^\infty(\mathbb{R}^{d})$, on can initially choose the countable subset ${\cal D}$ such that $\{ \nabla_v \Psi_2, \; \Psi_2 \in F \} \subset {\cal D}$, so that the conclusion of Proposition \ref{martingale_problem} holds true for all $\Psi(x,v) = \Psi_1(x) \Psi_2(v)$ with $\Psi_1 \in C_c^\infty(\mathbb{R}^d)$ and $\Psi_2 \in F$.
\end{remark}

\begin{proof}[Proof of Proposition \ref{martingale_problem}]
The martingale problem set on $\Omega$ may be expressed as 
\begin{align*}
& \E \Big[ \Big( M^n_\Psi(t) - M^n_\Psi(s) \Big) H^n \Big] = 0,
\\
& \E \Big[ \Big| M^n_\Psi(t) - M^n_\Psi(s) \Big|^2 H^n  \Big]
 = \E \Big[ \Big( V^n_\Psi(t) - V^n_\Psi(s)  \Big) H^n  \Big]
\end{align*}
for all $H^n =  h \Big( f^{R_n}(t_i), 1 \le i \le m \Big) $, where
$ 0 \le t_1, \ldots, t_m \le s \le t$,
and $h : (H^{-\sigma}_{W^{-1}}(\mathbb{R}^{2d}))^m \to \mathbb{R}$ is continuous and bounded,
Since the laws of $f^n$ and $\myhat f^n$ coincide, it follows that, on $\myhat \Omega$,
\begin{align}
& \myhat \E \Big[ \Big( \myhat M^n_\Psi(t) - \myhat M^n_\Psi(s) \Big) \myhat H^n \Big] = 0,
\label{chap3-mart1}
\\
& \myhat \E \Big[ \Big| \myhat M^n_\Psi(t) - \myhat M^n_\Psi(s) \Big|^2 \myhat H^n  \Big]
 = \myhat \E \Big[ \Big( \myhat V_\Psi^n(t) - \myhat V_\Psi^n(s)  \Big) \myhat H^n  \Big]
 \label{mart2}
\end{align}
where $ \myhat H^n =  h \Big( \myhat f^{n}(t_i), 1 \le i  \le m \Big)$ and $\myhat M^n_\Psi$ and $\myhat V^n_\Psi(t)$ are naturally defined on $\myhat \Omega$, as $\myhat M_\Psi$ and $\myhat V_\Psi(t)$ . We may decompose these into
\begin{align}
& \myhat M_\Psi^n(t) = \langle \Psi, \myhat f^{n}(t) \rangle - \langle \Psi, f_0 \rangle 
- \int_0^t \Big(  \Bigl< v \cdot \nabla_x \Psi, \myhat f^{n}(s) \Bigr> + \sum_{i=1}^4 {\cal M}^i_n(\myhat f^{n}(s))  \Big) ds,
\label{mart3}
\\
& \myhat V^n_\Psi(t) = \int_0^t \Big| {\cal M}^5(\myhat f^n(s)) \Big|^2 ds, \nonumber
\end{align}
where, recalling expressions \myref{reg},
\begin{align*}
 {\cal M}^1_n(f)  & = \int_x \int_v L^{CS}_{R_n} [f] \cdot \nabla_v \Psi f dx dv 
= \int_x \int_v (\Phi^1_n * f) \cdot \nabla_v \Psi f dx dv
, 
\\
 {\cal M}^2_n(f) &  = \int_x \int_v S_{R_n} [f] \cdot \nabla_v \Psi f dx dv 
 \nonumber \\
& = \frac{1}{2} \int_x \int_v \Big( \tld \psi * ( \Phi^2_n +  \Phi^3_n * f) - 
\| \tld \psi \|_{L^1}  ( \Phi^2_n +  \Phi^3_n * f) \Big) \cdot \nabla_v \Psi f dx dv
,
\\
 {\cal M}^3_n(f)  & = \frac{1}{2} \sum_{1 \le i,j \le d} \int_x \int_v K_{R_n}^i [f] K_{R_n}^j [f] \partial^2_{v_i v_j} \Psi f dx dv 
\nonumber \\
&  =  \frac{1}{2} \sum_{1 \le i,j \le d} \int_x \int_v 
\Big( \Phi^2_n + \Phi^3_n * f  \Big)^i  \Big( \Phi^2_n + \Phi^3_n * f \Big)^j  
\partial^2_{v_i v_j} \Psi f dx dv
, 
\\
 {\cal M}^4_n(f)  &=  \int_x \int_v  L^{MT}_R[f] \cdot \nabla_v \Psi f dx dv 
 = \int_x \int_v (u_R[f] - v) \cdot \nabla_v \Psi f dx dv
 \\ 
&  = \int_x  \frac{ \phi * \rho_{\theta_{R_n}}}{R^{-1} + \phi * \rho} \cdot \rho_{\nabla \Psi_2} \Psi_1 dx  - \int_x \int_v v \cdot \nabla_v \Psi f dx dv 
\\
{\cal M}^5_n(f) &= \int_x \int_v K_{R_n}[f] \cdot \nabla_v \Psi f dx dv
= \int_x \int_v (\Phi^2_n + \Phi^3_n * f) \cdot \nabla_v \Psi f dx dv 
\end{align*}
with
\begin{align*}
& \Phi^1_n(x,v) = \chi_{R_n}(x) \psi(x) \theta_{R_n}(-v),
\\
& \Phi^2_n(x) = \chi_{R_n}(x) F(x),
\\
& \Phi^3_n(x,v) = \chi_{R_n}(x) \tld \psi(x) \theta_{R_n}(-v),
\end{align*}
We wish to send $n$ to infinity in \myref{chap3-mart1} and \myref{mart2}. Thanks to the convergence \myref{skoro3}, the first three linear terms in \myref{mart3} cause no issue ; let us hence focus on the remaining terms. Let us consider the term involving ${\cal M}^1_n(f)$: defining the natural limiting term 
$${\cal M}^1(f) = \int_x \int_v (\Phi^1 * f) \cdot \nabla_v \Psi f dx dv$$ 
with $\Phi^1(x,v) = \psi(x) (-v)$, we have
\begin{align}
& \Big| \int_s^t {\cal M}^1_n(\myhat f^n(\sigma)) d\sigma - \int_s^t {\cal M}^1(\myhat f(\sigma)) d\sigma \Big|
\le {\cal J}^1_n + {\cal J}^2_n 
\label{M_split}
\end{align}
where
\begin{align*}
& {\cal J}^1_n =  \int_s^t \int_x \int_v \Big| (\Phi^1_n * \myhat f^n) - (\Phi^1 * \myhat f) \Big| | \nabla_v \Psi  || \myhat f | dx dv d\sigma,
\\
& {\cal J}^2_n =  \Big| \int_s^t \int_x \int_v  (\Phi^1_n * \myhat f^n) \cdot \nabla_v \Psi (\myhat f^n - \myhat f) dx dv d\sigma \Big|. 
\end{align*}
First, for ${\cal J}^1_n$ we have, thanks to \myref{uniform_e},
\begin{align}
\E \Big[ {\cal J}^1 \Big] & \lesssim \E \Big[ \int_s^t \int_{z \in Supp(\Psi)} \Big| (\Phi^1_n * \myhat f^n) - (\Phi^1 * \myhat f) \Big|^2 dz d\sigma \Big]^{1/2} \E \Big[ \int_s^t \int_z |\myhat f|^2 dz d\sigma \Big]^{1/2}
\nonumber \\ &
 \lesssim \E \Big[ \int_s^t \int_{z \in Supp(\Psi)} \Big| (\Phi^1_n * \myhat f^n) - (\Phi^1 * \myhat f) \Big|^2 dz d\sigma \Big]^{1/2}. \label{J1_1}
\end{align}
For all $n \ge m \ge 1$, for fixed $(x,v) \in Supp(\Psi)$, we may write
\begin{align*}
\myhat \E \Big[ \Big| \Phi^1_n * \myhat f^n - \Phi^1 * \myhat f \Big|^2  \Big] & \lesssim
\myhat  \E  \Big[ \Big| \Phi^1_m * \myhat f^n - \Phi^1_m * \myhat f \Big|^2  \Big]
+\myhat  \E \Big[ \Big| (\Phi^1_n - \Phi^1_m) * \myhat f^n  \Big|^2  \Big]
\\ & \hspace{5mm}
+ \myhat  \E \Big[ \Big| (\Phi^1 - \Phi^1_m) * \myhat f \Big|^2 \Big]
\end{align*}
which converges to $0$ as $n$ goes to infinity thanks to the convergence \myref{skoro3}, and the bounds
\begin{align*}
\myhat \E \Big[ \Big| (\Phi^1_n - \Phi^1_m) * \myhat f^n  \Big|^2  \Big] & \lesssim
\myhat  \E \Big[ \Big| \int \int_{|v-w| \ge R_m} |v-w| f^n(y,w) dy dw \Big|^2 \Big] 
\\ &
 \lesssim R_m^{-2} \Big( |v|^4 + \E \Big[ \int |w|^4 f^n dy dw \Big] \Big)
  \lesssim R_m^{-2} (1 + |v|^4)
\end{align*}
and, for some $\gamma > 0$, recalling that $(x,v) \in Supp(\Psi)$,
\begin{align*}
\myhat \E \Big[ \Big| \Phi_m^1 * \myhat f^n \Big|^{2 + \gamma}(x,v)  \Big] \lesssim 1 + \myhat \E \Big[  \int |w|^{2 + \gamma} \myhat f^n dy fw  \Big] \lesssim 1
\end{align*}
for all $n \ge m$. Note that this last bound also guarantees the uniform integrability in $(\omega, \sigma, x, v)$ of the integrand in \myref{J1_1}, so that $\myhat \E \Big[ {\cal J}^1_n \Big] \to 0$. 
Additionally, for $\gamma >0$ small enough, the bound
\begin{align*}
& \myhat \E \Big[ \int_s^t \int_{z \in Supp(\Psi)} \Big| \Phi^1_n * \myhat f^n  \Big|^{2 + \gamma} |\myhat f|^{2+\gamma} dz d\sigma  \Big] 
\\
& \hspace{30mm} \lesssim 
\myhat \E \Big[ \int_s^t \int_{z \in Supp(\Psi)} \Big( 1 + \int |w|^{2+ \gamma} \myhat f^n dy dw  \Big) |\myhat f|^{2+\gamma} dz d\sigma  \Big]
\nonumber \\
&  \hspace{50mm} \lesssim \E \Big[\int_s^t \int_{\mathbb{R}^{2d}} |v|^k \myhat f^n + |\myhat f|^p dz d\sigma \Big] \lesssim 1 
\end{align*}
guarantees that
$
\E \Big[ \Big| {\cal J}^1_n \Big|^{2 + \gamma} \Big] \lesssim 1.
$
Similarly, for ${\cal J}^2_n$, we write, for all $n \ge m \ge 1$,
\begin{align*}
{\cal J}^2_n \le &  \Big| \int_s^t \int_x \int_v (\Phi^1_m * \myhat f^m) \cdot \nabla_v \Psi (\myhat f^n - \myhat f) dx dv d\sigma \Big| 
\\
& \; \;  + \int_s^t \int_x \int_v \Big|(\Phi^1_m * \myhat f^m) - (\Phi^1_n * \myhat f^n) \Big| |\nabla_v \Psi | (|\myhat f^n| + |\myhat f|) dx dv d\sigma
\end{align*}
so that
\begin{align*}
\E \Big[ {\cal J}_n^2 \Big]  \lesssim & \; \E \Big[ \Big| \int_s^t \int_x \int_v \Big( (\Phi^1_m * \myhat f^m) \cdot \nabla_v \Psi \Big) (\myhat f^n - \myhat f) dx dv d\sigma \Big|  \Big] 
\\
&
+ \E \Big[ \int_s^t \int_{z \in Supp(\Psi)} \Big| (\Phi^1_m * \myhat f^m) - (\Phi^1_n * \myhat f^n) \Big|^2 dz d\sigma  \Big].
\end{align*}
Since, for fixed $m \ge 1$, $(\Phi^1_m * \myhat f^m) \cdot \nabla_v \Psi \in C_c^\infty(\mathbb{R}^{2d})$, we conclude in a similar fashion that $\E \Big[ {\cal J}^2_n \Big] \to 0$ and 
$\E \Big[ \Big| {\cal J}^2_n \Big|^{2 + \gamma} \Big] \lesssim 1$. Coming back to \myref{M_split}, we have shown that
\begin{align*}
& \int_s^t {\cal M}^1_n(\myhat f^n(\sigma)) d\sigma \to \int_s^t  {\cal M}^1(\myhat f(\sigma)) d\sigma \text{ in probability},
\\
& \myhat \E \Big[ \Big|  \int_s^t {\cal M}^1_n(\myhat f^n(\sigma)) d\sigma  \Big|^{2 + \gamma} \Big] \lesssim 1,
\end{align*}
which is sufficient to pass to the limit in the corresponding term of \myref{chap3-mart1} and the left-hand side of \myref{mart2}. 

The terms involving ${\cal M}^i_n(f)$ for $i=2,3,5$ can be treated with similar arguments.
Let us now handle the more delicate term, involving ${\cal M}^4_n(f)$: we wish to prove that
\begin{align}
\int_s^t \int_x \frac{ \phi * \myhat \rho^n_{\theta_{R_n}}}{ R_n^{-1} + \phi * \myhat \rho^n} \cdot \myhat \rho^n_{\nabla \Psi_2} \Psi_1 dx d\sigma = 
\int_s^t \int_x u_{R}[\myhat f^n]  \cdot \myhat \rho^n_{\nabla \Psi_2} \Psi_1 dx d\sigma
\label{delicate_term}
\end{align}
converges to the expected limiting term. Let us introduce
\begin{align}
\myhat J^n := u_{R_n}[\myhat f^n]  \cdot  \myhat \rho^n_{\nabla \Psi_2} = u_{R_n}[ \myhat f^n] \cdot \int_v  \nabla_v \Psi_2(v) \myhat f^n dv.
\label{Jn}
\end{align}
For some $q > 2$ and $\tau \in (0,1)$, we have
\begin{align*}
\int |\myhat J^n|^q dx  \lesssim \int_x \int_v  |u_{R_n}[\myhat f^n] |^q | \myhat f^n |^q dx dv 
& \lesssim \int_x \int_v |u_{R_n}[\myhat f^n] |^{\frac{q}{1-\tau}} \myhat f^n dx dv + \int_x \int_v |\myhat f^n|^{1 + \frac{q-1}{\tau}} dx dv
\\
& \lesssim \int_x \int_v (1 + |u_{R_n}[\myhat f^n] |^k) \myhat f^n dx dv + \int_x \int_v | \myhat f^n |^p dx dv
\end{align*} 
as soon as $ \tau = \theta (q-1)$ and $\frac{q}{1-\tau} \le \frac{4}{1- \theta}$, which can be met for $q$ close enough to $2$. Hence, thanks to Proposition \ref{ineq_u}, we derive
\begin{align}
\myhat \E \Big[ \sup_{t \in [0,T]} \int_x |\myhat J^n(t,x)|^q dx \Big] \lesssim \E \Big[ \sup_{t \in [0,T]} \Big( \int_z ( 1 + |v|)^k f^n(t,z) dz + \int_z  |f^n(t,z)|^p dz \Big) \Big]  \lesssim 1.
\label{bound_J}
\end{align}
As a consequence, up to some subsequence which we omit for simplicity,
\begin{align}
\myhat J^n \rightharpoonup \myhat J \; \; \text{ weak $*$ in} L^q \Big( \myhat \Omega ; L^\infty([0,T] ; L^q(\mathbb{R}^d)) \Big).
\label{weak_star}
\end{align}
It is clear that this weak convergence is enough for the term \myref{delicate_term} to pass to the limit in \myref{chap3-mart1}. Therefore, it only remains to identify the limit as
\begin{align}
\myhat J = \myhat u \cdot \myhat \rho_{\nabla \Psi_2}, \;  \text{ where }  \;
\myhat u(x,t) := u[\myhat f](x,t) = 
\left\{
\begin{array}{l l}
\displaystyle{ \frac{( \phi * \myhat j)(t,x)}{ (\phi * \myhat \rho) (t,x)} } & \text{ if } (\phi * \myhat \rho) (t,x) \neq 0
\\
0 & \text{ if } (\phi * \myhat \rho) (t,x) = 0.
\end{array}
\right.
\label{limit_u}
\end{align}
First, considering the set, for $r > 0$, $\omega \in \myhat \Omega$,
$$
A_r(\omega) = \{ (t,x) \in [0,T] \times B(0,r) \; | \; (\phi * \myhat \rho)(\omega,t,x) = 0 \},
$$
we have
\begin{align}
\int_{A_r} | \myhat J^n | dx dt & \lesssim \Big( \int_{(t,x) \in A_r} \int_{v \in Supp(\Psi_2)} |u_{R_n}[\myhat f^n] |^2 \myhat f^n dx dt \Big)^{1/2} \Big( \int_{A_r} \myhat \rho^n dx dt \Big)^{1/2}
\label{bound_J_2}
\end{align}
so that, with Proposition \ref{ineq_u},
\begin{align*}
\myhat \E \Big[ \int_{A_r(\omega)} |\myhat J^n| dx dt \Big] \lesssim
 \E \Big[  \int_0^T \int_{\mathbb{R}^{2d}} |v ]|^2 f^{R_n} dz dt  \Big]^{1/2}
 \myhat \E \Big[ \int_{A_r(\omega)} \myhat \rho^n dx dt \Big]^{1/2} \lesssim 
  \myhat \E \Big[ \int_{A_r(\omega)} \myhat \rho^n dx dt \Big]^{1/2}.
\end{align*}
From \myref{weak_star} and Lemma \ref{rho_j}, we deduce
\begin{align*}
\myhat \E \Big[ \int_0^T \int_{A_r(\omega)     }  |\myhat J| dx dt \Big] \lesssim \myhat \E \Big[\int_0^T \int_{ \{ x \in B(0,r) \; | \; \phi * \myhat \rho = 0 \}    }  \myhat \rho dx dt \Big]^{1/2} = 0.
\end{align*}
Indeed, it is easy to see that, when $\rho \in C(\mathbb{R}^d)$,
$$
 \int_{ \{ x \in B(0,r) \; | \; \phi * \myhat \rho = 0 \} }  \rho(x) dx = 0
$$
and the same equality is deduced for any $\rho \in L^2(\mathbb{R}^d)$ by density. Hence, $\myhat \Proba$ almost surely, $\myhat J = 0$ a.e on $A_r$. Since this holds for any $r > 0$, we deduce that, $\myhat \Proba$ almost surely, $\myhat J = 0$ holds whenever $\phi * \myhat \rho = 0$, so that we only have to check equality \myref{limit_u} whenever $( \phi * \myhat \rho) (t,x) \neq 0$.


Recalling 	\myref{ae_as_1} and Lemma \ref{rho_j}, we have
\begin{align*}
\myhat J^n = \frac{ \phi * \myhat \rho^n_{\theta_{R_n}}}{ R_n^{-1} + \phi * \myhat \rho^n} \cdot \myhat \rho^n_{\nabla \Psi_2} \to 
\frac{ \phi * \myhat j}{  \phi * \myhat \rho} \cdot \myhat \rho_{\nabla \Psi_2} = \myhat u \cdot \rho_{\nabla \Psi_2}
\end{align*}
almost everywhere on the set $\{ (t,x) \in [0,T] \times \mathbb{R}^d \; | \;  \phi * \myhat \rho \neq 0 \}$, $\myhat \Proba$ almost surely. It follows that equality \myref{limit_u} holds.

Finally, it remains to pass \myref{delicate_term} to the limit in the quadratic equality \myref{mart2}. To this intent, we simply notice that
\begin{align*}
\Big| \int_s^t \int_x \frac{ \phi * \myhat \rho^n_{\theta_{R_n}}}{ R_n^{-1} + \phi * \myhat \rho^n} \cdot \myhat \rho^n_{\nabla \Psi_2} \Psi_1 dx d\sigma \Big|^2 = 
\int_{\sigma_1} \int_{\sigma_2} \int_x \int_y \myhat J^n(\sigma_1, x) \myhat J^n(\sigma_2, y) \Psi_1(x) \Psi_2(y) dx dy d\sigma_1 d\sigma_2
\end{align*}
where $\myhat J^n(\sigma,x)$ is defined as in \myref{Jn}. We have the uniform bound
$$
\myhat \E \Big[ \Big| \sup_{\sigma_1 \in [0,T]} \sup_{\sigma_2 \in [0,T]} \int_x \int_y 
\Big|
\myhat J^n(\sigma_1,x) \myhat J^n(\sigma_2,y)  
  \Big|^q dx dy \Big|^{1/2}  \Big] = 
  \myhat \E \Big[ \Big| \sup_{\sigma \in [0,T]}\int_x \Big|
\myhat J^n(\sigma,x)  \Big|^q dx  \Big] \lesssim 1
$$
with $q/2 > 1$ so that, up to some subsequence
\begin{align*}
\myhat J^n(\sigma_1,x) \myhat J^n(\sigma_2,y)  \rightharpoonup
\myhat J(\sigma_1, \sigma_2, x, y)
 \text{ weak $*$ in } L^{q/2}\Big( \myhat \Omega ; L^\infty([0,T]^2 ; L^q((\mathbb{R}^d)^2)) \Big).
\end{align*}
This time, thanks to \myref{bound_J}
\begin{align*}
& \myhat \E \Big[ \int_{(x, \sigma_1) \in A_r} \int_{(y,\sigma_2) \in B(0,r) \times [0,T]} \Big| \myhat J^n(\sigma_1,x) \myhat J^n(\sigma_2,y) \Big| dx d\sigma_1 dy d\sigma_2  \Big]
\\
& \hspace{20mm} \lesssim \myhat \E \Big[ \Big( \int_{A_r} \Big| \myhat J^n \Big| dx dt \Big)^2 \Big]^{1/2} 
\myhat \E \Big[ \int_0^T \int_{B(0,r)} \Big| \myhat J^n \Big|^2 dx dt \Big]^{1/2}
\lesssim \myhat \E \Big[ \Big( \int_{A_r} \Big| \myhat J^n \Big| dx dt \Big)^2 \Big]^{1/2}.
\end{align*}
Using \myref{bound_J_2}, it follows that
\begin{align*}
\myhat \E \Big[ \Big( \int_{A_r} \Big| \myhat J^n \Big| dx dt \Big)^2 \Big] &
\lesssim \E \Big[ \int_0^T \int_{\mathbb{R}^{2d}} |u_{R_n}[f^{R_n}]|^4 |f^{R_n}|^2 dz dt \Big]^{1/2}
\myhat \E \Big[ \int_{A_r} |\myhat \rho^n|^2 dx dt \Big]^{1/2}
\\
& \lesssim \E \Big[ \int_0^T \int_{\mathbb{R}^{2d}} |v|^k f^{R_n} + |f^{R_n}|^p dz dt \Big]^{1/2}
\myhat \E \Big[ \int_{A_r} |\myhat \rho^n|^2 dx dt \Big]^{1/2}
\\
& \lesssim \myhat \E \Big[ \int_{A_r} | \myhat \rho^n|^2 dx dt \Big]^{1/2}
\end{align*}
Using again Lemma \ref{rho_j}, we deduce that, $\myhat \Proba$ almost surely, $\myhat J = 0$ almost everywhere on
$$
\left\{ (\sigma_1, \sigma_2, x,y) \in [0,T]^2 \times (\mathbb{R}^d)^2 \; | \; (\phi*\myhat \rho)(\sigma_1,x) = 0 \text{ or } (\phi*\myhat \rho)(\sigma_2,y) = 0 \right\}.
$$
The limit is then determined to be $\myhat J(\sigma_1, \sigma_2, x,y) = \Big( \myhat u(x) \cdot \rho_{\nabla \Psi_2}(\sigma_1,x) \Big) \Big( \myhat u(y) \cdot \rho_{\nabla \Psi_2}(\sigma_2,y) \Big)$ on the complementary set using pointwise convergence, as done previously.

\end{proof}

From the martingale problem of Proposition \ref{martingale_problem}, we classically construct a martingale solution of \myref{chap3-spde}, in the sense of Definition \ref{martingale_def}, using a martingale representation theorem. 
First, note that estimates 
\myref{mart_est1} and \myref{mart_est2} are easily derived from the convergence \myref{skoro3} using Fatou's Lemma. Introducing the process
\begin{align}
\myhat M(t) = \myhat f(t) - f_0 - \int_0^t ({\cal L}[\myhat f(s)])^* \myhat f(s) ds, \; \; \; t \in [0,T]
\label{H_mart}
\end{align}
we see that, for all test function $\Psi \in C^\infty_c(\mathbb{R}^{2d})$ of the form \myref{form_psi},
$$
\Bigl< \myhat M(t), \Psi \Bigr> = \myhat M_\Psi(t), \; \; \; t \in [0,T]
$$
which is a continuous $L^2$ martingale with respect to the filtration $(\myhat {\cal F}_t)_{t \in [0,T]}$, with quadratic variation $\myhat V_\Psi$. With Remark \ref{rem_psi} in mind, 
by density, we may carefully extend this statement to any test function $\Psi$ in some separable Hilbert space ${\cal H}$. One may convince oneself for instance that the weighted Sobolev space
$$
{\cal H} = \left\{ \Psi \in D'(\mathbb{R}^{2d}), \; \; \max_{|\beta| \le 2} \int_{\mathbb{R}^{2d}} 
e^{|z|}  | \partial^\beta_z \Psi(z) |^2 dz \right\}
$$
is suitable.
Using a polarization formula, we deduce that for $\Psi_1, \Psi_2 \in {\cal H}$,
$$
\Bigl< \myhat M(t) , \Psi_1 \Bigr> \Bigl< \myhat M(t) , \Psi_2 \Bigr> - \Bigl< \myhat V(t) \Psi_1, \Psi_2  \Bigr>,
\; \; \; t \in [0,T]
$$
defines a continuous $(\myhat {\cal F}_t)_{t \in [0,T]}$-martingale, where the operator $\myhat V(t)$ is defined through
$$
\Bigl< \myhat V(t) \Psi_1, \Psi_2  \Bigr> = \int_0^t 
\Bigl<  K[\myhat f(s)] \cdot \nabla_v \Psi_1, \myhat f(s) \Bigr>
\Bigl<  K[\myhat f(s)] \cdot \nabla_v \Psi_2, \myhat f(s) \Bigr> ds.
$$
The martingale representation theorem from \cite{da_prato} (p222, Theorem 9.2) holds for the ${\cal H}'$-valued process \myref{H_mart}, giving another probability space $(\hat \Omega, \hat {\cal F}, \hat \Proba)$ equipped with a filtration $(\hat {\cal F}_t)_{t \in [0,T]}$, and a $(\myhat {\cal F}_t \times \hat {\cal F}_t)$-brownian motion $(W_t)_{t \in [0,T]}$ on 
$\myhat \Omega \times \hat \Omega$ such that
$$
\myhat M(t) (\myhat \omega, \hat \omega) := \myhat M(t) (\myhat \omega) = - \int_0^t 
\nabla_v \cdot \Big( K[\myhat f(s)] \myhat f(s) \Big) dW_s(\myhat \omega, \hat \omega).
$$
It follows that $(\myhat \omega, \hat \omega) \in \myhat \Omega \times \hat \Omega \mapsto (\myhat f(t)(\myhat \omega))_{t \in [0,T]}$ defines a solution of \myref{chap3-spde} on $\myhat \Omega \times \hat \Omega$.
This concludes the proof of Theorem \ref{chap3-thm1}.

\end{subsection}

\begin{subsection}{Strong local alignment}
The proof of Theorem \ref{chap3-thm2} can be established using the same arguments. Indeed, the only role played by the weight function $\phi$ in the estimates of sections \ref{sec:unif_estimates} and \ref{sec:av_lemma} is through the constant $C(\phi)$ in Proposition \ref{ineq_u}, given by \myref{C_phi}. Introducing $\phi_r$ of the form
$
\phi_r(x) = r^d \phi_1(x/r)
$
we notice that 
$$ 
C(\phi_r) \propto \frac{ r^d \sup_{B(0,r_2)} \phi_1  }{ \inf_{B(0,r_1)} \phi_1} \Big( \frac{r_2}{r_1 \times r} \Big)^d
=  \frac{\sup_{B(0,r_2)} \phi_1  }{ \inf_{B(0,r_1)} \phi_1} \Big( \frac{r_2}{r_1} \Big)^d \propto C(\phi_1)
$$
uniformly in $r > 0$. As a result, introducing $f^r$ a weak solution of \myref{chap3-spde} with $\phi=\phi^r$ constructed (in law) as previously, the following estimates hold uniformly in $r > 0$: for any $\varphi \in C_c^\infty(\mathbb{R}^d)$,
\begin{align*}
& \E \Big[ \sup_{t \in [0,T]} \int_{\mathbb{R}^{2d}} |f^r(t,z)|^p + (1 + |x|^{2} + |v|^k) f^r(t,z) dz  \Big] 
+ \E \Big[ \| \rho^r_\varphi \|^2_{L^2_t H^\eta_x}  \Big]  \lesssim 1
\end{align*}
and for some $q > 1$,
\begin{align*}
& E \Big[ \| f^r(t) - f^r(s) \|^{2q}_{H^{-2}}  \Big] +  
\E \Big[ \| \rho^r(t) - \rho^r(s) \|^{2q}_{H^{-2}}  \Big]  \le |t-s|^q.
\end{align*}
From here, we may use the same arguments as previously to establish the tightness of $(f^r)_{r > 0}$  
(resp. $(\rho^r_\varphi)_{r > 0}$) in 
$C([0,T] ; H^{-\sigma}_{W^{-1}}(\mathbb{R}^{2d}))$ for $\sigma > 0$
(resp. in $L^2([0,T] ; L^2(\mathbb{R}^d))$) and then pass to the limit in the martingale problem satisfied by $f^{r_n}$ as $n$ goes to infinity.

\end{subsection}

\end{section}

\nocite{*}
\bibliographystyle{plain}
\bibliography{stochastic_MT}


\end{document}